\documentclass[12pt,oneside,a4papre]{amsart}

\usepackage{amssymb}
\usepackage{amsmath}
\usepackage{amsthm}
\usepackage{amscd}

\usepackage{longtable}
\usepackage{mathrsfs}

\usepackage[dvipdfmx]{xcolor}
\usepackage[dvipdfmx]{pict2e}
\usepackage[dvipdfmx]{graphicx}
\usepackage{comment}
\usepackage{todonotes}

\usepackage[foot]{amsaddr}

\usepackage{tikz}

\usepackage[all]{xy}

\usetikzlibrary{cd}

\theoremstyle{plain}
\newtheorem{theorem}{Theorem}[section]
\newtheorem{lemma}[theorem]{Lemma}
\newtheorem{corollary}[theorem]{Corollary}
\newtheorem{proposition}[theorem]{Proposition}
\newtheorem{remark}[theorem]{Remark}

\theoremstyle{definition}
\newtheorem{definition}[theorem]{Definition}

\newtheorem{conjecture}[theorem]{Conjecture}

\newcommand{\kah}{K\"{a}hler }
\newcommand{\idd}{i\partial\overline{\partial}}


\begin{document}
\title
[Nakano-Nadel type, Bogomolov-Sommese type vanishing] 
{Nakano-Nadel type, Bogomolov-Sommese type vanishing and singular dual Nakano semi-positivity}
\author{Yuta Watanabe}
\date{}
\address{The University of Tokyo, Komaba, Meguro-ku, Tokyo, Japan}
\email{watayu@g.ecc.u-tokyo.ac.jp, wyuta.math@gmail.com}

\begin{abstract}
   In this article, we get properties for singular (dual) Nakano semi-positivity and obtain vanishing theorems involving $L^2$-subsheaves on weakly pseudoconvex manifolds by $L^2$-estimates and $L^2$-type Dolbeault isomorphisms.
   As applications, Fujita's conjecture type theorem with singular Hermitian metrics is presented.
\end{abstract}

\vspace{-5mm}

\maketitle

\tableofcontents

\vspace{-8mm}

\section{Introduction}\label{section:1}


Throughout this paper, we let $X$ be an $n$-dimensional complex manifold. 
Let $\varphi$ be a plurisubharmonic function on $X$ and let $\mathscr{I}(\varphi)$ be the sheaf of germs of holomorphic functions $f$ such that $|f|^2e^{-\varphi}$ is locally integrable which is called the multiplier ideal sheaf.
For a singular Hermitian metric $h$ on holomorphic line bundles, we define the multiplier ideal sheaf by $\mathscr{I}(h)\!=\!\mathscr{I}(\varphi)$ where $h\!=\!e^{-\varphi}$ locally. 
As an invariant of the singularities of the plurisubharmonic functions, the multiplier ideal sheaf play important role in the study of the several complex variables and complex algebraic geometry.

In vanishing theorems of the type involving multiplier ideal sheaf, the Nadel-Demailly vanishing theorem [5,\,28] 
is well known as an extension of the Kodaira vanishing theorem \cite{Kod53}, and the Kawamata-Viehweg vanishing theorem (cf. [21,\,37] 
and [7,\,Theorem\,6.25]) is a more detailed judgment of the positivity.
These theorems are for $(n,q)$-forms, but recently a version for $(p,n)$-forms was presented in \cite{Wat22b}, which was then given in \cite{LMNWZ22} with a more detailed judgment of positivity analogous to Kawamata-Viehweg. 

That is, the following: Let $X$ be a projective manifold, $\omega$ be a \kah metric on $X$ and $L$ be a holomorphic line bundle equipped with a singular Hermitian metric $h$. 
If $(L,h)$ is big, i.e. $i\Theta_{L,h}\geq\varepsilon\omega$
in the sense of currents for some $\varepsilon>0$. 
Then we have that 
\begin{align*}
    H^q(X,K_X\otimes L\otimes \mathscr{I}(h))&=0  \quad \mathrm{for~any} \quad q>0 \quad (\mathrm{see\, [5,\,28]}),\\
    H^n(X,\Omega_X^p\otimes L\otimes\mathscr{I}(h))&=0 \quad \mathrm{for~any} \quad p>0 \quad (\mathrm{see\, [39]}).
\end{align*}

Note that the above vanishing theorem cannot be extended to the same bidegree $(p,q)$ with $p+q>n$ as the Kodaira-Akizuki-Nakano type vanishing theorem by the Ramanujam's counterexample (cf. \cite{Ram72},\,[39,\,Remark\,2.10]). 

Recently, vanishing theorems for $(n,q)$-forms involving (Demailly) $m$-positive holomorphic vector bundles and multiplier ideal sheaves were shown 
in \cite{Hua20} on compact \kah manifolds by using $L^2$-Hogde isomorphisms. 

In this paper, we first define a $\it{dual}$ $m$-$\it{positivity}$ (see Definition \ref{Def of dual m-posi}) corresponding to $(p,n)$-forms of $m$-positivity and obtain the following vanishing theorems involving (dual) $m$-positive holomorphic vector bundles and multiplier ideal sheaves by using $L^2$-estimates. 
Here, a function $\psi:X\to[-\infty,+\infty)$ is $\it{exhaustive}$ if all sublevel sets $X_c:=$ $\{z\in X\mid \psi(x)<c\},\,c<\sup_X\psi$, are relatively compact.
A complex manifold is said to be $\it{weakly}$ $\it{pseudoconvex}$ if there exists a smooth exhaustive plurisubharmonic function.

\begin{theorem}\label{Ext NDBS V-thm for big for (n,q)}
    Let $X$ be a weakly pseudoconvex \kah manifold, $F$ be a holomorphic vector bundle of rank $r$ over $X$ and $L$ be a holomorphic line bundle over $X$ equipped with a singular Hermitian metric $h$.
    We assume that $F$ is $m$-semi-positive and $h$ is singular positive in the sense of Definition \ref{def of psef & big}.
    Then we have that
        \begin{align*}
            H^q(X,K_X\otimes F\otimes L\otimes\mathscr{I}(h))=0
        \end{align*}
        for $q>0$ with $m\geq\min\{n-q+1,r\}$.
\end{theorem}

If the singular Hermitian metric on line bundle possesses semi-positive, then the following holds by assuming $m$-positivity for the vector bundle.

\begin{theorem}\label{Ext NDBS V-thm for psef for (n,q)}
    Let $X$ be a weakly pseudoconvex manifold and $L$ be a holomorphic line bundle over $X$ equipped with a singular semi-positive Hermitian metric $h$, i.e. $i\Theta_{L,h}\geq0$ in the sense of currents.
    Then we have the following
    \begin{itemize}
        \item [$(a)$] If $X$ has a \kah metric and $A$ is a $k$-positive holomorphic line bundle then 
        \begin{align*}
            H^q(X,K_X\otimes A\otimes L\otimes\mathscr{I}(h))=0 \qquad \mathit{for~any} ~\,q\geq k.
        \end{align*}
        \item [$(b)$] If $F$ is a $m$-positive holomorphic vector bundle of rank $r$ then 
        \begin{align*}
            H^q(X,K_X\otimes F\otimes L\otimes\mathscr{I}(h))=0
        \end{align*}
        for $q>0$ with $m\geq\min\{n-q+1,r\}$.
    \end{itemize}
\end{theorem}

Here, Theorem \ref{Ext NDBS V-thm for big for (n,q)} and $(b)$ of Theorem \ref{Ext NDBS V-thm for psef for (n,q)} are generalizations to weakly pseudoconvex of Theorem 1.14 and Theorem 1.9 in \cite{Hua20}, respectively.
For $(p,n)$-forms, the following is obtained by using $L^2$-estimates in $\S$4 and an $L^2$-type Dolbeault isomorphism in $\S$5.

\begin{theorem}\label{Ext NDBS V-thm for big for (p,n)}
    Let $X$ be a compact \kah manifold, $F$ be a holomorphic vector bundle of rank $r$ over $X$ and $L$ be a holomorphic line bundle over $X$ equipped with a singular Hermitian metric $h$.
    We assume that $F$ is dual $m$-semi-positive and $h$ is singular positive in the sense of Definition \ref{def of psef & big}.
    Then we have that
        \begin{align*}
            H^n(X,\Omega^p_X\otimes F\otimes L\otimes\mathscr{I}(h))=0
        \end{align*}
        for $p>0$ with $m\geq\min\{n-p+1,r\}$.
\end{theorem}

\begin{theorem}\label{Ext NDBS V-thm for psef for (p,n)}
    Let $X$ be a compact manifold and $L$ be a holomorphic line bundle over $X$ equipped with a singular semi-positive Hermitian metric $h$, i.e. $i\Theta_{L,h}\geq0$ in the sense of currents.
    Then we have the following
    \begin{itemize}
        \item [$(a)$] If $X$ has a \kah metric and $A$ is a $k$-positive holomorphic line bundle then 
        \begin{align*}
            H^n(X,\Omega^p_X\otimes A\otimes L\otimes\mathscr{I}(h))=0  \qquad \mathit{for~any} ~\,p\geq k.
        \end{align*}
        \item [$(b)$] If $F$ is a dual $m$-positive holomorphic vector bundle of rank $r$ then 
        \begin{align*}
            H^n(X,\Omega_X^p\otimes F\otimes L\otimes\mathscr{I}(h))=0
        \end{align*}
        for $p>0$ with $m\geq\min\{n-p+1,r\}$.
    \end{itemize}
\end{theorem}

Note, $n$-th cohomology always vanishes on non-compact complex manifolds (cf. [27,\,30]). 

Notions of singular Hermitian metrics for holomorphic vector bundles were introduced and investigated (cf. [2,\,3]) 
and of positivity for singular Hermitian metrics is very interesting subjects. It is known that we cannot always define the curvature currents with measure coefficients \cite{Rau15}. 
Therefore, semi-negativity and semi-positivity for Griffiths and (dual) Nakano (cf. [2,\,11,\,19,\,31,\,33,\,39]) 
were defined naturally without using the curvature currents by using the properties of plurisubharmonicity.

We study properties of singular (dual) Nakano semi-positivity and vanishing theorems. 
Among them we obtain the following dual-type generalization of Demailly and Skoda's theorem (cf. [9,\,26]) 
to singularities. This theorem is already known for $L^2$-type (dual) Nakano semi-positivity (see. [19,\,Theorem\,1.3], [39,\,Theorem\,5.3]), but is presented here for a more natural and stronger definition. 

\begin{theorem}\label{Grif * det Grif is dual Nakano}
    Let $X$ be a complex manifold and $E$ be a holomorphic vector bundle over $X$ equipped with a singular Hermitian metric $h$.
    If $h$ is Griffiths semi-positive then $h\otimes\mathrm{det}\,h$ is dual Nakano semi-positive.
\end{theorem}

We also get the following vanishing theorems which are generalizations of Griffiths type vanishing theorem to singularities, (dual) $m$-positivity and weakly pseudoconvex \kah manifolds.
Here, $\mathscr{E}(h)$ is the $L^2$-subsheaf of $\mathcal{O}_X(E)$ with respect to a singular Hermitian metric $h$ on $E$ analogous to multiplier ideal sheaves. In fact, $\mathscr{E}(h)=\mathcal{O}_X(E)\otimes\mathscr{I}(h)$ if $E$ is a holomorphic line bundle.
And if $h$ is Griffiths semi-positive, then it is already known in \cite{HI20} and \cite{Ina22} that $\mathscr{E}(h\otimes\mathrm{det}\,h)$ is coherent.

\begin{theorem}\label{V-thm of Griffiths semi-posi + (dual)-m-posi for (n,q)}
    Let $X$ be a weakly pseudoconvex manifold 
    and $E$ be a holomorphic vector bundle over $X$ equipped with a singular Hermitian metric $h$.
    We assume that $h$ is Griffiths semi-positive on $X$.
    Then we have the following
    \begin{itemize}
        \item [$(a)$] If $X$ has a positive holomorphic line bundle and $A$ is a $k$-positive holomorphic line bundle, then for any $q\geq k$ we have that
        \begin{align*}
            H^q(X,K_X\otimes A\otimes \mathscr{E}(h\otimes\mathrm{det}\,h))=0. 
        \end{align*}
        \item [$(b)$] If $F$ is a $m$-positive holomorphic vector bundle of rank $r$ over $X$ then 
        \begin{align*}
            H^q(X,K_X\otimes F\otimes \mathscr{E}(h\otimes\mathrm{det}\,h))=0
        \end{align*}
        for $q\geq1$ and $m\geq\min\{n-q+1,r\}$.
    \end{itemize}
\end{theorem}

\begin{theorem}\label{V-thm of Griffiths semi-posi + (dual)-m-posi for (p,n)}
    Let $X$ be a projective manifold and $E$ be a holomorphic vector bundle over $X$ equipped with a singular Hermitian metric $h$.
    We assume that $h$ is Griffiths semi-positive on $X$.
    Then we have the following
    \begin{itemize}
        \item [$(a)$] If $A$ is a $k$-positive holomorphic line bundle, then for any $p\geq k$ we have that
        \begin{align*}
            H^n(X,\Omega^p_X\otimes A\otimes \mathscr{E}(h\otimes\mathrm{det}\,h))=0.  
        \end{align*}
        \item [$(b)$] If $F$ is a dual $m$-positive holomorphic vector bundle of rank $r$ over $X$ then 
        \begin{align*}
            H^n(X,\Omega^p_X\otimes F\otimes \mathscr{E}(h\otimes\mathrm{det}\,h))=0
        \end{align*}
        for $p\geq1$ and $m\geq\min\{n-p+1,r\}$.
    \end{itemize}
\end{theorem}

Moreover, in $\S$6, we provide vanishing theorems for singular (dual) Nakano semi-positivity twisted by (dual) $m$-positive vector bundles on weakly pseudoconvex manifolds with a positive holomorphic line bundle (resp. projective manifolds). 

Recently, a Fujita Conjecture type theorem involving multiplier ideal sheaves was presented in \cite{SY19} using vanishing theorems.
Finally, we get a vanishing theorem that is a more detailed judgment of positivity by numerically dimension for nef line bundle on projective manifolds.
And we obtain Fujita's conjecture type theorem involving the $L^2$-subsheaf as an application of our vanishing theorems.

\begin{theorem}\label{V-thm of nu(N) and Grif for nu(h)<1 near A}
    Let $X$ be a compact \kah manifold and $E$ be a holomorphic vector bundle equipped with a singular Hermitian metric $h$. 
    Let $N$ be a nef line bundle which is neither big nor numerically trivial, i.e. $\nu(N)\notin\{0,n\}$. 
    If $h$ is Griffiths semi-positive and there exists a smooth ample divisor $A$ such that $\nu(-\log\mathrm{det}\,h|_A,x)<1$ for all points in $A$ and that $\nu(N|_A)=\nu(N)$, then we have that 
    \begin{align*}
        H^q(X,K_X\otimes N\otimes \mathscr{E}(h\otimes\mathrm{det}\,h))=0
    \end{align*}
    for $q>n-\nu(N)$.
\end{theorem}

\begin{theorem}\label{Fujita Conj for Grif}
    Let $X$ be a compact \kah manifold and $E$ be a holomorphic vector bundle equipped with a singular Hermitian metric $h$. 
    Let $L$ be an ample and globally generated line bundle and $N$ be a nef but not numerically trivial line bundle.
    If $h$ is Griffiths semi-positive and there exists a smooth ample divisor $A$ such that $\nu(-\log\mathrm{det}\,h|_A,x)<1$ for all points in $A$ and that $\nu(N|_A)=\nu(N)$,  
    then the coherent sheaf
    \begin{align*}
        K_X\otimes L^{\otimes n}\otimes N\otimes \mathscr{E}(h\otimes\mathrm{det}\,h)
    \end{align*}
    is globally generated.
\end{theorem}

\section{Smooth Hermitian metrics and positivity}

In this section, we define various positivity for holomorphic vector bundles and show equivalence relations with Nakano semi-positivity by using $L^2$-estimates.

Let $\omega$ be 
 a Hermitian metric on $X$ and $(E,h)$ be a holomorphic Hermitian vector bundle of rank $r$ over $X$.
Let $D=D'+\overline{\partial}$ be the Chern connection of $(E,h)$, and $\Theta_{E,h}=[D',\overline{\partial}]=D'\overline{\partial}+\overline{\partial}D'$ be the Chern curvature tensor. 
Let $(U,(z_1,\cdots,z_n))$ be local coordinates. Denote by $(e_1,\cdots,e_r)$ an orthonormal frame of $E$ over $U\subset X$, and
\begin{align*}
    i\Theta_{E,h,x_0}=i\sum_{j,k}\Theta_{jk}dz_j\wedge d\overline{z}_k=i\sum_{j,k,\lambda,\mu}c_{jk\lambda\mu}dz_j\wedge d\overline{z}_k\otimes e^*_\lambda\otimes e_\mu,~\,\, \overline{c}_{jk\lambda\mu}=c_{kj\mu\lambda}.
\end{align*}

To $i\Theta_{E,h}$ corresponds a natural Hermitian form $\theta_{E,h}$ on $T_X\otimes E$ defined by 
\begin{align*}
    \theta_{E,h}(u):=\theta_{E,h}(u,u)&=\sum c_{jk\lambda\mu}u_{j\lambda}\overline{u}_{k\mu},~\,\,\mathrm{for}\,\,~ u=\sum u_{j\lambda}\frac{\partial}{\partial z_j}\otimes e_\lambda\in T_{X,x_0}\otimes E_{x_0},\\
    \mathrm{i.e.}~\,\,\, \theta_{E,h}&=\sum c_{jk\lambda\mu}(dz_j\otimes e^*_\lambda)\otimes\overline{(dz_k\otimes e^*_\mu)}.
\end{align*}

\begin{definition}\label{def of k-posi}
    Let $L$ be a holomorphic line bundle over a complex manifold $X$. 
    We say that $L$ is $k$-$\it{positive}$ if there exists a smooth Hermitian metric $h$ such that $i\Theta_{L,h}$ is semi-positive and has at least $n-k+1$ positive eigenvalues at every point of $X$.
\end{definition}

\begin{definition}$(\mathrm{cf.\,[8,\,ChapterVII],\,[26,\,Definition\,2.1]})$ 
    Let $T$ and $E$ be complex vector spaces of dimensions $n,r$ respectively, and $\Theta$ be a hermitian form on $T\otimes E$.
    Let $(E,h)$ be a holomorphic vector bundle over a complex manifold $X$.
    \begin{itemize}
        \item A tensor $u\in T\otimes E$ is said to be of rank $m$ if $m$ is the smallest $\geq0$ integer such that $u$ can be written $u=\sum^m_{j=1}\xi_j\otimes s_j$, where $\xi_j\in T,~\,s_j\in E$.
        \item $\Theta$ is $m$-$\it{positive}$ (resp. $m$-$\it{semi}$-$\it{positive}$) if $\Theta(u)>0$ (resp. $\Theta(u)\geq0$) for any tensor $0\ne u\in T\otimes E$ of rank $\leq m$. In this case, we write $\Theta>_m0$ (resp. $\geq_m0$).
        \item $(E,h)$ is $m$-$\it{positive}$ (resp. $m$-$\it{semi}$-$\it{positive}$) if $\theta_{E,h}>_m0$ (resp. $\theta_{E,h}\geq_m0$). 
        \item $(E,h)$ is said to be $\it{Griffiths ~positive}$ (resp. $\it{Griffiths ~semi}$-$\it{positive}$) if $(E,h)$ is $1$-positive (resp. $1$-semi-positive).
        \item $(E,h)$ is said to be $\it{Nakano ~positive}$ (resp. $\it{Nakano ~semi}$-$\it{positive}$) if $\theta_{E,h}$ is positive (resp. semi-positive) definite as a Hermitian form on $T_X\otimes E$, 
        i.e. $\theta_{E,h}(u)>0$ (resp. $\geq0$).
        Here, Nakano positivity corresponds to $m\geq\min\{n,r\}$.
        \item $(E,h)$ is said to be $\it{dual~Nakano ~positive}$ (resp. $\it{dual~Nakano ~semi}$-$\it{positive}$) if $(E^*,h^*)$ is Nakano negative (resp. Nakano semi-negative).
    \end{itemize}
\end{definition}

It is clear that the concepts of Griffiths positive, Nakano positive, $1$-positive and positive coincide if $\mathrm{rank}\,E=1$.
We introduce another notion about $m$-positivity that correspond to positivity for $(p,n)$-forms.

\begin{definition}\label{Def of dual m-posi}
    Let $X$ be a complex manifold of dimension $n$ and $(E,h)$ be a holomorphic Hermitian vector bundle of rank $r$ over $X$.
    $(E,h)$ is said to be $\it{dual}$ $m$-$\it{positive}$ (resp. $\it{dual}$ $m$-$\it{semi}$-$\it{positive}$) 
    if $(E^*,h^*)$ is $m$-$\it{negative}$ (resp. $m$-$\it{semi}$-$\it{negative}$). 
\end{definition}

Here, $E$ is said to be $\bullet$-$\it{positive}$ (resp. $\bullet$-$\it{semi}$-$\it{positive}$) if there exists a smooth Hermitian metric $h_E$ such that $(E,h_E)$ is $\bullet$-$\it{positive}$ (resp. $\bullet$-$\it{semi}$-$\it{positive}$), where $\bullet$ is (dual) $m$, Griffiths and Nakano.
Notes that $1$-positivity and dual $1$-positivity are equivalent due to the equivalence between Griffiths-positivity of $(E,h)$ and Griffiths-negativity of $(E^*,h^*)$.

We denote the curvatur operator $[i\Theta_{E,h},\Lambda_\omega]$ on $\Lambda^{p,q}T^*_X\otimes E$ by $A^{p,q}_{E,h,\omega}$ and the fact that $[i\Theta_{E,h},\Lambda_\omega]$ is positive (resp. semi-positive) definite on $\Lambda^{p,q}T^*_X\otimes E$ is simply written as $A^{p,q}_{E,h,\omega}>0$ (resp. $\geq0$).
We obtain the following lemma for the relationship between positivity of curvature operator $A^{p,q}_{E,h,\omega}$ and (dual) $m$-positivity by using [8,\,ChapterVII,\,Lemma\,7.2] and [38,\,Theorem\,2.3\,\,and\,\,2.5]. 

\begin{lemma}\label{m-posi then A>0 and dual m-posi then A>0}
    Let $(X,\omega)$ be a Hermitian manifold and $(E,h)$ be a holomorphic vector bundle over $X$. Then we obtain the following
    \begin{itemize}
        \item [($a$)] If $(E,h)$ is $m$-positive (resp. $m$-semi-positive), then we get 
        \begin{align*}
            A^{n,q}_{E,h,\omega}=[i\Theta_{E,h},\Lambda_\omega]>0\quad(resp.~\geq0)\quad \it{for}~\, q\geq1 ~\,and~\, m\geq\min\{n-q+1,r\}.
        \end{align*}
        \item [($b$)] If $(E,h)$ is dual $m$-positive (resp. dual $m$-semi-positive), then we get 
        \begin{align*}
            A^{p,n}_{E,h,\omega}=[i\Theta_{E,h},\Lambda_\omega]>0\quad(resp.~\geq0)\quad \it{for}~\, p\geq1 ~\,and~\, m\geq\min\{n-p+1,r\}.
        \end{align*}
    \end{itemize}
\end{lemma}


\begin{proposition}\label{A_(E+F)>C_E+C_F if A_E>C_E and A_F>C_F}
    Let $(X,\omega)$ be a Hermitian manifold of dimension $n$ and $p,q$ be fixed integers. Let $(E,h_E)$ and $(F,h_F)$ be holomorphic vector bundles over $X$ and $C_E,C_F$ be non-negative real numbers.
    If $A^{p,q}_{E,h_E,\omega}\geq C_E$ and $A^{p,q}_{F,h_F,\omega}\geq C_F$, then we have that $A^{p,q}_{E\otimes F,h_E\otimes h_F,\omega}\geq C_E+C_F$.
\end{proposition}

\begin{proof}
    Let $x_0\in X$ and $(z_1,\ldots,z_n)$ be local coordinates such that $(\partial/\partial z_1,\ldots,\partial/\partial z_n)$ is an orthonormal basis of $(T_X,\omega)$ at $x_0$. 
    Let $(e_1,\ldots,e_r)$ and $(f_1,\ldots,f_r)$ be orthonormal bases of $E_{x_0}$ and $F_{x_0}$, respectively. We can write $\omega_{x_0}=i\sum^n_{j=1} dz_j\wedge d\overline{z}_j$ and
    \begin{align*}
        i\Theta_{E,h_E,x_0}=i\sum c^E_{jk\lambda\mu}dz_j\wedge d\overline{z}_k\otimes e^*_\lambda\otimes e_\mu,\,\,
        i\Theta_{F,h_F,x_0}=i\sum c^F_{jk\lambda\mu}dz_j\wedge d\overline{z}_k\otimes f^*_\lambda\otimes f_\mu.
    \end{align*}
    Let $J,K$ be ordered multi-indices with $|J|=p$ and $|K|=q$.
    For any $(p,q)$-form $ u\in\Lambda^{p,q}T^*_{X,x_0}\otimes E_{x_0}\otimes F_{x_0}$, we can write
    \begin{align*}
        u=\sum_{|J|=p,|K|=q,\lambda,\tau}u_{JK\lambda\tau}dz_J\wedge d\overline{z}_K\otimes e_\lambda\otimes f_\tau=\sum_\tau u^E_\tau\otimes f_\tau=\sum_\lambda u^F_\lambda \otimes e_\lambda,
    \end{align*}
    where $u^E_\tau=\sum_{|J|=p,|K|=q,\lambda}u_{JK\lambda\tau}dz_J\wedge d\overline{z}_K\otimes e_\lambda,\, u^F_\lambda=\sum_{|J|=p,|K|=q,\tau}u_{JK\lambda\tau}dz_J\wedge d\overline{z}_K\otimes f_\tau$.
    As is well known, we have that 
    \begin{align*}
        \Lambda_\omega u&=i(-1)^p\sum_{J,K,\lambda,\tau,s}u_{JK\lambda\tau}\Bigl(\frac{\partial}{\partial z_s}\lrcorner~dz_J\Bigr)\wedge \Bigl(\frac{\partial}{\partial \overline{z}_s}\lrcorner~ d\overline{z}_K\Bigr)\otimes e_\lambda\otimes f_\tau\\
        &=\sum_\tau \bigl(\Lambda_\omega u^E_\tau\bigr)\otimes f_\tau=\sum_\lambda \bigl(\Lambda_\omega u^F_\lambda\bigr)\otimes e_\lambda,
    \end{align*}        
    \begin{align*}
        i\Theta_{E,h_E}\otimes \mathrm{id}_F u&=i\sum_{j,k,\lambda,\mu,\tau}\bigl(c^E_{jk\lambda\mu}dz_j\wedge d\overline{z}_k\otimes e^*_\lambda\otimes e_\mu\otimes f^*_\tau \otimes f_\tau\bigr) u\\
        &=\sum_\tau \Bigl(\bigl(i\Theta_{E,h_E}\bigr)\otimes f^*_\tau \otimes f_\tau \Bigr)\Bigl(\sum_\alpha u^E_\alpha\otimes f_\alpha\Bigr)=\sum_\tau \bigl(i\Theta_{E,h_E}u^E_\tau\bigr)\otimes f_\tau.
    \end{align*}
    Therefore, we get
    \begin{align*}
        [i\Theta_{E,h_E}\otimes \mathrm{id}_F,\Lambda_\omega]u&=\bigl(i\Theta_{E,h_E}\otimes \mathrm{id}_F\bigr) \wedge\Lambda_\omega u-\Lambda_\omega \wedge \bigl(i\Theta_{E,h_E}\otimes \mathrm{id}_F\bigr) u\\
        &=\sum_\tau \Bigl(\bigl(i\Theta_{E,h_E}\bigr)\otimes f^*_\tau \otimes f_\tau \Bigr)\wedge\Bigl(\sum_\alpha \bigl(\Lambda_\omega u^E_\alpha\bigr)\otimes f_\alpha\Bigr)\\
        &\qquad\qquad -\Lambda_\omega \wedge \Bigl(\sum_\tau \bigl(i\Theta_{E,h_E}u^E_\tau\bigr)\otimes f_\tau\Bigr)\\
        &=\sum_\tau \bigl(i\Theta_{E,h_E}\wedge \Lambda_\omega \bigr)u^E_\tau \otimes f_\tau - \sum_\tau \bigl(\Lambda_\omega\wedge i\Theta_{E,h_E}\bigr) u^E_\tau \otimes f_\tau\\
        &=\sum_\tau \Bigl([i\Theta_{E,h_E},\Lambda_\omega]u^E_\tau\Bigr)\otimes f_\tau.
    \end{align*}

    Hence, we obtain the following 
    \begin{align*}
        \langle A^{p,q}_{E\otimes F,h_E\otimes h_F,\omega}u,u\rangle_{h_E\otimes h_F,\omega}
        &=\langle[i\Theta_{E,h_E}\otimes \mathrm{id}_F,\Lambda_\omega]u,u\rangle_{h_E\otimes h_F,\omega}+\langle[i\Theta_{F,h_F}\otimes \mathrm{id}_E,\Lambda_\omega]u,u\rangle_{h_E\otimes h_F,\omega}\\
        &=\Bigl\langle \sum_\tau \Bigl([i\Theta_{E,h_E},\Lambda_\omega]u^E_\tau\Bigr)\otimes f_\tau, \sum_\alpha u^E_\alpha\otimes f_\alpha \Bigr\rangle_{h_E\otimes h_F,\omega} \\
        &\qquad\qquad +\Bigl\langle \sum_\lambda \Bigl([i\Theta_{F,h_F},\Lambda_\omega]u^F_\lambda\Bigr)\otimes e_\lambda, \sum_\beta u^F_\beta\otimes e_\beta \Bigr\rangle_{h_E\otimes h_F,\omega}\\
        &=\sum_\tau \langle A^{p,q}_{E,h_E,\omega}u^E_\tau,u^E_\tau\rangle_{h_E,\omega}+\sum_\lambda \langle A^{p,q}_{F,h_F,\omega}u^F_\lambda,u^F_\lambda\rangle_{h_F,\omega}\\
        &\geq \sum_\tau C_E|u_\tau|^2_{h_E,\omega}+\sum_\lambda C_F|u_\lambda|^2_{h_F,\omega} 
        =(C_E+C_F)|u|^2_{h_E\otimes h_F,\omega}.
    \end{align*}
    This represents $A^{p,q}_{E\otimes F,h_E\otimes h_F,\omega}\geq C_E+C_F$.
\end{proof}

Finally, we present an equivalence relation for smooth Nakano semi-positive Hermitian metric using $L^2$-estimates.
This characterization by $L^2$-estimates for Nakano semi-positivity has been described in [11,\,Theorem\,1.1] and [19,\,Proposition\,2.8], and was also characterized for Nakano semi-negativity in [38,\,Theorem\,1.7]. 
For these characterizations, Proposition \ref{characterization of smooth Nakano} is a generalization involving the invariance of positivity in tensor products, i.e. the tensor product of Nakano semi-positive vector bundle and $m$-semi-positive vector bundle is also $m$-semi-positive.

\begin{proposition}\label{characterization of smooth Nakano}
    Let $E$ be a holomorphic vector bundle and $h$ be a smooth Hermitian metric on $E$. 
    Then the following conditions are equivalent.
    \begin{itemize}
        \item [(1)] $h$ is Nakano semi-positive, i.e. $A^{n,q}_{E,h,\omega}\geq0$ for any $q\geq1$ and \kah metric $\omega$.
        \item [(2)] For any positive integer $k\in\{1,\cdots,n\}$, any Stein coordinate $S$, any K\"ahler metric $\omega_S$ on $S$ and
        any smooth Hermitian metric $h_F$ on any holomorphic vector bundle $F$ such that $A^{n,s}_{F,h_F,\omega_S}>0$ for $s\geq k$,
        we have that for any $q\geq k$ and any $\overline{\partial}$-closed $f\in L^2_{n,q}(S,E\otimes F,h\otimes h_F,\omega_S)$, there exists $u\in L^2_{n,q-1}(S,E\otimes F,h\otimes h_F,\omega_S)$ satisfying $\overline{\partial}u=f$ and 
        \begin{align*}
            \int_S|u|^2_{h\otimes h_F,\omega_S}dV_{\omega_S}\leq\int_S\langle B^{-1}_{h_F,\omega_S}f,f\rangle_{h\otimes h_F,\omega_S}dV_{\omega_S},
        \end{align*}
        provided the right-hand side is finite, where $B_{h_F,\omega_S}=[i\Theta_{F,h_F}\otimes\mathrm{id}_E,\Lambda_{\omega_S}]$.
    \end{itemize}
\end{proposition}

\begin{proof}
    First, we show $(1) \Longrightarrow (2)$. Here, $h$ is Nakano semi-positive if and only if $A^{n,q}_{E,h,\omega_S}\geq0$ for $q\geq1$. From the proof of Proposition \ref{A_(E+F)>C_E+C_F if A_E>C_E and A_F>C_F}, we have that
    \begin{align*}
        A^{n,q}_{E\otimes F,h\otimes h_F,\omega_S}&=[i\Theta_{E,h}\otimes\mathrm{id}_F,\Lambda_{\omega_S}]+[i\Theta_{F,h_F}\otimes\mathrm{id}_E,\Lambda_{\omega_S}]\\
        &\geq [i\Theta_{F,h_F}\otimes\mathrm{id}_E,\Lambda_{\omega_S}]=B_{h_F,\omega_S}>0
    \end{align*}
    on $S$ for any $q\geq k$.
    Since [8,\,ChapterVIII,\,Theorem\,6.1], i.e. $L^2$-estimates for $(n,q)$-forms without completeness of \kah metrics, 
    for any $q\geq k$ and any $\overline{\partial}$-closed $f\in L^2_{n,q}(S,E\otimes F,h\otimes h_F,\omega_S)$, there exists $u\in L^2_{n,q-1}(S,E\otimes F,h\otimes h_F,\omega_S)$ such that $\overline{\partial}u=f$ and that
    \begin{align*}
        \int_S|u|^2_{h\otimes h_F,\omega}dV_{\omega_S}&\leq\int_S\langle(A^{n,q}_{E\otimes F,h\otimes h_F,\omega_S})^{-1}f,f\rangle_{h\otimes h_F,\omega_S}dV_{\omega_S}\\
        &\leq\int_S\langle B^{-1}_{h_F,\omega_S}f,f\rangle_{h\otimes h_F,\omega_S}dV_{\omega_S}.
    \end{align*}

    Second, we consider $(2) \Longrightarrow (1)$. 
    We take a Stein coordinate $S$, a \kah metric $\omega$ on $S$ and a holomorphic vector bundle $F$.
    Let $h_F$ be a smooth Hermitian metric on $F$ such that $A^{n,s}_{F,h_F,\omega}>0$ for $s\geq k$. 
    Then for any $q\geq k$ and any $\overline{\partial}$-closed $f\in \mathcal{E}^{n,q}(S,E\otimes F)\subset L^2_{n,q}(S,E\otimes F,h\otimes h_F,\omega)$, there is $u\in L^2_{n,q-1}(S,E\otimes F,h\otimes h_F,\omega)$ such that $\overline{\partial}u=f$ and 
    \begin{align*}
        ||u||^2_{h\otimes h_F,\omega}=\int_S|u|^2_{h\otimes h_F,\omega}dV_\omega\leq\int_S\langle B^{-1}_{h_F,\omega}f,f\rangle_{h\otimes h_F,\omega}dV_\omega,
    \end{align*}
    where we assume that the right-hand side is finite.

    From the Bochner-Kodaira-Nakano identity, for any $\alpha\in\mathcal{E}^{n,q}(S,E\otimes F)$  
    we have that
    \begin{align*}
        |\langle\langle f,\alpha\rangle\rangle_{h\otimes h_F,\omega}|^2&=|\langle\langle\overline{\partial}u,\alpha\rangle\rangle_{h\otimes h_F,\omega}|^2=|\langle\langle u,\overline{\partial}^*_{h\otimes h_F}\alpha\rangle\rangle_{h\otimes h_F,\omega}|^2\\
        &\leq ||u||^2_{h\otimes h_F,\omega}||\overline{\partial}^*_{h\otimes h_F}\alpha||^2_{h\otimes h_F,\omega}\\
        &\leq \int_S\langle B^{-1}_{h_F,\omega}f,f\rangle_{h\otimes h_F,\omega}dV_\omega\cdot\Bigl(||D'^*_{h\otimes h_F}\alpha||^2_{h\otimes h_F,\omega}\\
        &\qquad+\langle\langle[i\Theta_{E,h}\otimes\mathrm{id}_F,\Lambda_\omega]\alpha,\alpha\rangle\rangle_{h\otimes h_F,\omega}+\langle\langle B_{h_F,\omega}\alpha,\alpha\rangle\rangle_{h\otimes h_F,\omega}\Bigr),
    \end{align*}
    where $D'_{h\otimes h_F}$ is the $(1,0)$ part of the Chern connection on $E\otimes F$ with respect to the metric $h\otimes h_F$. 
    Let $\alpha=B^{-1}_{h_F,\omega}f$, i.e. $f=B_{h_F,\omega}\alpha$. Then the above inequality becomes
    \begin{align*}
        |\langle\langle B_{h_F,\omega}\alpha,\alpha\rangle\rangle_{h\otimes h_F,\omega}|^2
        &\leq \langle\langle \alpha,B_{h_F,\omega}\alpha\rangle\rangle_{h\otimes h_F,\omega}\bigl(||D'^*_{h\otimes h_F}\alpha||^2_{h\otimes h_F,\omega}\\
        &\qquad+\langle\langle[i\Theta_{E,h}\otimes\mathrm{id}_F,\Lambda_\omega]\alpha,\alpha\rangle\rangle_{h\otimes h_F,\omega}+\langle\langle B_{h_F,\omega}\alpha,\alpha\rangle\rangle_{h\otimes h_F,\omega}\bigr).
    \end{align*}

    Therefore we get
    \begin{align*}
        \langle\langle[i\Theta_{E,h}\otimes\mathrm{id}_F,\Lambda_\omega]\alpha,\alpha\rangle\rangle_{h\otimes h_F,\omega}+||D'^*_{h\otimes h_F}\alpha||^2_{h\otimes h_F,\omega}\geq0. \tag{$\ast$}
    \end{align*}
    Using this formula $(\ast)$, we show the proposition by contradiction.

    Suppose that $A^{n,q}_{E,h,\omega}$ is not semi-positive for any $q\geq1$. We fixed this integer $q\geq1$. Then there is $x_0\in X$ and $\xi_0\in \Lambda^{n,q}T^*_{X,x_0}\otimes E_{x_0}$ such that 
    \begin{align*}
        \langle[i\Theta_{E,h},\Lambda_\omega]\xi_0,\xi_0\rangle_{h,\omega}=-2c
    \end{align*}
    for some $c>0$. 

    For any $k\in\{1,\cdots,n\}$ and any $R>0$, we define the following Stein subsets of $\mathbb{C}^n$;
    \begin{center}
        $B^k_R:=\{(z_1,\cdots,z_{n-k+1})\in\mathbb{C}^{n-k+1}\mid \sum_{j=1}^{n-k+1}|z_j|^2<R\}\subset\mathbb{C}^{n-k+1}$,

        $D^k_R:=\{(z_{n-k+2},\cdots,z_n)\in\mathbb{C}^{k-1}\mid \sum_{j=n-k+2}^n|z_j|^2<R\}\subset\mathbb{C}^{k-1},\quad\,$
    \end{center}
    where $B^1_R:=B_R$.
    Let $(B_{2R},(z_1,\cdots,z_n))$ be a holomorphic coordinate in $X$ centered at $x_0$ and $\omega=\idd|z|^2$ be a \kah metric on $B_{2R}$. Here, we can choose $R$ such that $E|_{B_{2R}}$ is trivial on $B_{2R}$. 
    For any trivial holomorphic vector bundle $F=B_{2R}\times\mathbb{C}^t$ where $t=\mathrm{rank}\,F$, let $I_F$ be a trivial Hermitian metric on $F$.
    Fixed an integer $k$ with $q\geq k$. We define the Stein subset $S_R:=B^k_R\times D^k_R\subset B_{2R}$ and define the plurisubharmonic function $\psi_k:=\sum_{j=1}^{n-k+1}|z_j|^2-R^2/4$ which is $k$-convex and define the smooth Hermitian metric $I_F^{m\psi_k}$ by $I_Fe^{-m\psi_k}$. Then for any $s\geq k$, we get
    \[
        A^{n,s}_{F,I_F^{m\psi_k},\omega}=[i\Theta_{F,I_F^{m\psi_k}},\Lambda_\omega]=m[\idd\psi_k\otimes\mathrm{id}_F,\Lambda_\omega]\geq m>0.
    \]

    Let $e=(e_1,\cdots,e_r), b=(b_1,\cdots,b_t)$ be holomorphic frames of $E,F$ respectively, where $b$ is orthonormal frame with respect to $I_F$.
    For any $u=\sum u_{J\lambda}dz_N\wedge d\overline{z}_J\otimes e_\lambda\in \mathcal{E}^{n,q}(B_{2R},E)$, let $u_F=\sum u\otimes b_\tau\in \mathcal{E}^{n,q}(B_{2R},E\otimes F)$. Then we have that 
    \begin{align*}
        B_{I_F^{m\psi_k},\omega}u_F&
        =m[\idd\psi_k\otimes\mathrm{id}_{E\otimes F},\Lambda_\omega]u_F=m\sum_\tau\Bigl([\idd\psi_k\otimes\mathrm{id}_{E},\Lambda_\omega]u\Bigr)\otimes b_\tau\\
        &=m\sum_\tau\Bigl(\Bigl(\sum_{j\in J\cap I_k}1\Bigr)u_{J\lambda}dz_N\wedge d\overline{z}_J\otimes e_\lambda \Bigr)\otimes b_\tau\\
        &=m\sum_{J,\lambda,\tau}|J\cap I_k|u_{J\lambda}dz_N\wedge d\overline{z}_J\otimes e_\lambda\otimes b_\tau,\quad\mathrm{and}\\
        B^{-1}_{I_F^{m\psi_k},\omega}u_F&=\frac{1}{m}\sum_{J,\lambda,\tau}|J\cap I_k|^{-1}u_{J\lambda}dz_N\wedge d\overline{z}_J\otimes e_\lambda\otimes b_\tau,
    \end{align*}
    where $I_k=\{1,\cdots,n-k+1\},\,J\cap I_k\ne\emptyset$ and $|J\cap I_k|:=\#\{J\cap I_k\}$.

    Let $\xi=\sum \xi_{J\lambda}dz_N\wedge d\overline{z}_J\otimes e_\lambda\in\mathcal{E}^{n,q}(B_{2R},E)$ with constant coefficients such that $\xi(x_0)=\xi_0$ and let $\xi_F:=\sum_\tau \xi\otimes b_\tau$. 
    We may assume
    \begin{align*}
        \langle[i\Theta_{E,h},\Lambda_\omega]\xi,\xi\rangle_{h,\omega}<-c
    \end{align*}
    on $B_{2R}$, for any small number $R>0$.

    For any ordered multi-index $I$, we define $\varepsilon(s,I)\in\{-1,0,1\}$ (see [38,\,Definition\,2.1]) 
    as the number that satisfies $\xi_s\lrcorner\,\xi^*_I=\varepsilon(s,I)\xi^*_{I\setminus s}$, where $(\xi_1,\ldots,\xi_n)$ is an orthonormal basis of $T_X$.
    Here, the symbol $\bullet\lrcorner\,\bullet$ represents the interior product, i.e. $\xi_s\lrcorner\,\xi^*_I=\iota_{\xi_s}\xi^*_I$.

    Choose $\chi_B\in\mathscr{D}(B^k_R,\mathbb{R}_{\geq 0})$ such that $\chi_B|_{B^k_{R/2}}=1$ and that $\mathrm{supp}\,\chi_B\subset B^k_{3R/4}$. Let 
    \begin{align*}
        v=\sum_{J,\lambda}\sum_{1\leq j\leq n-k+1}(-1)^n\varepsilon(j,J)\xi_{J\lambda}\overline{z}_j\chi_B(z)dz_N\wedge d\overline{z}_{J\setminus j}\otimes e_\lambda\in\mathcal{E}^{n,q-1}(S_R,E),
    \end{align*}
    then from $(-1)^n\varepsilon(j,J)d\overline{z}_j\wedge dz_N\wedge d\overline{z}_{J\setminus j}=dz_N\wedge d\overline{z}_J$, we have that 
    \begin{align*}
        \overline{\partial}v_F|_{G_{R/2}}&=\sum_\tau\overline{\partial}v|_{G_{R/2}}\otimes b_\tau
        =\sum_{J,\lambda,\tau}\sum_{1\leq j\leq n-k+1}\sum_{1\leq l\leq n}(-1)^n\varepsilon(j,J)\xi_{J\lambda}\delta_{jl}d\overline{z}_l\wedge dz_N\wedge d\overline{z}_{J\setminus j}\otimes e_\lambda\otimes b_\tau\\
        &=\sum_{J,\lambda,\tau}\sum_{1\leq j\leq n-k+1}(-1)^n\varepsilon(j,J)\xi_{J\lambda}d\overline{z}_j\wedge dz_N\wedge d\overline{z}_{J\setminus j}\otimes e_\lambda\otimes b_\tau\\
        &=\sum_{J,\lambda,\tau}\sum_{j\in J\cap I_k}\xi_{J\lambda} dz_N\wedge d\overline{z}_J\otimes e_\lambda\otimes b_\tau
        =B_{I_F^{\psi_k},\omega}\xi_F,
    \end{align*}
    where $G_{R/2}:=B^k_{R/2}\times D^k_R$ and $j\notin J$ then $\varepsilon(j,J)=0$.

    Let $f:=\overline{\partial}v\in\mathcal{E}^{n,q}(S_R,E)$ and $f_F:=\sum_\tau f\otimes b_\tau=\sum_\tau\overline{\partial}v\otimes b_\tau=\overline{\partial}v_F\in\mathcal{E}^{n,q}(S_R,E\otimes F)$ then we get $\overline{\partial}f_F=0$ on $S_R$ and $f_F=B_{I_F^{\psi_k},\omega}\xi_F$ with constant coefficients on $G_{R/2}$.
    We define $\alpha_m:=B^{-1}_{I_F^{m\psi_k},\omega}f_F=\frac{1}{m}B^{-1}_{I_F^{\psi_k},\omega}f_F\in\mathcal{E}^{n,q}(S_R,E\otimes F)$, where $\alpha_m|_{G_{R/2}}=\frac{1}{m}\xi_F$. Here,
    \begin{align*}
        f_F=\chi_B(z)B_{I_F^{\psi_k},\omega}\xi_F+\sum_{J,\lambda,\tau}\sum_{j\in J\cap I_k,l\in I_k} \xi_{J\lambda}\overline{z}_j\frac{\partial\chi_B(z)}{\partial\overline{z}_l}(-1)^n\varepsilon(j,J)d\overline{z}_l\wedge dz_N\wedge d\overline{z}_{J\setminus j}\otimes e_\lambda\otimes b_\tau.
    \end{align*}

    Since $v$ depends only on the variables $z_1,\cdots,z_{n-k+1}$, then $f_F=\overline{\partial}v_F$ is also depends only on the variables $z_1,\cdots,z_{n-k+1}$.
    By smoothness of $h$ on $X$, $h$ is bounded on $S_R$. Hence, from $\chi_B$ and $\psi_k$ depend only on the variables $z_1,\cdots,z_{n-k+1}$ and $\mathrm{supp}\,f_F\subset\mathrm{supp}\,\chi_B\subset\subset B^k_{3R/4}\times D_R^k$, we have that 
    \begin{align*}
        \int_{S_R}\langle B^{-1}_{I_F^{m\psi_k},\omega}f_F,f_F\rangle_{h\otimes I_F,\omega}e^{-m\psi_k}dV_\omega=\int_{S_R}\frac{1}{m}\langle B^{-1}_{I_F^{\psi_k},\omega}f_F,f_F\rangle_{h\otimes I_F,\omega}e^{-m\psi_k}dV_\omega<+\infty,
    \end{align*}
    for any $m>0$.
    From $i[\Lambda_\omega,\overline{\partial}]=D'^*_{h\otimes I_F^{m\psi_k}}$ (cf. [7,\,Chapter\,4]), we have that 
    \begin{align*}
        D'^*_{h\otimes I_F^{m\psi_k}}\alpha_m&=\frac{1}{m}D'^*_{h\otimes I_F^{m\psi_k}}\xi_F=0, \quad\mathrm{and}\\
        \langle[i\Theta_{E,h}\otimes\mathrm{id}_F,\Lambda_\omega]\alpha_m,\alpha_m\rangle_{h\otimes I_F,\omega}&=\frac{1}{m^2}\langle[i\Theta_{E,h}\otimes\mathrm{id}_F,\Lambda_\omega]\xi_F,\xi_F\rangle_{h\otimes I_F,\omega}\\
        &=\frac{1}{m^2}\langle\sum_\tau([i\Theta_{E,h}\otimes\mathrm{id}_F,\Lambda_\omega]\xi)\otimes b_\tau,\sum_\sigma\xi\otimes b_\sigma\rangle_{h\otimes I_F,\omega}\\
        &=\frac{1}{m^2}\langle[i\Theta_{E,h},\Lambda_\omega]\xi,\xi\rangle_{h,\omega}\sum_{\tau,\sigma}\langle b_\tau,b_\sigma\rangle_{I_F}
        <
        -\frac{c}{m^2}\mathrm{rank}\,F
    \end{align*}
    on $G_{R/2}$, where $\langle b_\tau,b_\sigma\rangle_{I_F}=\delta_{\tau,\sigma}$. Since $f_F$ depends only on the variables $z_1,\cdots,z_{n-k+1}$ and $\mathrm{supp}\,f_F\subset\subset B^k_{3R/4}\times D_R^k\subset S_R$, there is a constant $C$, such that 
    \begin{align*}
        |\langle[i\Theta_{E,h}\otimes\mathrm{id}_F,\Lambda_\omega]\alpha_m,\alpha_m\rangle_{h\otimes I_F,\omega}|\leq\frac{C}{m^2},\,\,|D'^*_{h\otimes I_F^{m\psi_k}}\alpha_m|^2_{h\otimes I_F,\omega}\leq\frac{C}{m^2}
    \end{align*}
    on $S_R$ for any $m>0$.

    Then we consider the left-hand side of $(\ast)$ with respect to $(S,h_F,\alpha)=(S_R,I_F^{m\psi_k},\alpha_m)$.
    \begin{align*}
        m^2\Bigl(\langle\langle[&i\Theta_{E,h}\otimes\mathrm{id}_F,\Lambda_\omega]\alpha_m,\alpha_m\rangle\rangle_{h\otimes I_F^{m\psi_k},\omega}+||D'^*_{h\otimes I_F^{m\psi_k}}\alpha_m||^2_{h\otimes I_F^{m\psi_k},\omega}\Bigr)\\
        &=m^2\Bigl(\int_{G_{R/2}}\langle[i\Theta_{E,h}\otimes\mathrm{id}_F,\Lambda_\omega]\alpha_m,\alpha_m\rangle_{h\otimes I_F,\omega}e^{-m\psi_k}dV_\omega\\
        &\qquad+\int_{S_R\setminus G_{R/2}}\langle[i\Theta_{E,h}\otimes\mathrm{id}_F,\Lambda_\omega]\alpha_m,\alpha_m\rangle_{h\otimes I_F,\omega}e^{-m\psi_k}dV_\omega\\
        &\qquad+\int_{S_R\setminus G_{R/2}}|D'^*_{h\otimes I_F^{m\psi_k}}\alpha_m|^2_{h\otimes I_F,\omega}e^{-m\psi_k}dV_\omega\Bigr)\\
        &\leq-c\cdot\mathrm{rank}\,F\int_{G_{R/2}}e^{-m\psi_k}dV_\omega+2C\int_{S_R\setminus G_{R/2}}e^{-m\psi_k}dV_\omega\\
        &= \mathrm{Vol}(D^k_R)\Bigl(-c\cdot\mathrm{rank}\,F\int_{B^k_{R/2}}e^{-m\psi_k}dV_{\omega_k}+2C\int_{B^k_R\setminus B^k_{R/2}}e^{-m\psi_k}dV_{\omega_k}\Bigr),
    \end{align*}
    where $\omega_k=\idd\psi_k$.
    Since $\lim_{m\to+\infty}m\psi_k(z)=+\infty$ for $z\in B_R^k\setminus \overline{B}_{R/2}^k$, and $\psi_k(z)\leq0$ for $z\in B_{R/2}^k$. Therefore we obtain that 
    \begin{align*}
        \langle\langle[i\Theta_{E,h}\otimes\mathrm{id}_F,\Lambda_\omega]\alpha_m,\alpha_m\rangle\rangle_{h\otimes I_F^{m\psi_k},\omega}+||D'^*_{h\otimes I_F^{m\psi_k}}\alpha_m||^2_{h\otimes I_F^{m\psi_k},\omega}<0
    \end{align*}
    for $m>>1$, which contradicts to the inequality $(\ast)$.
\end{proof} 

\section{Singular Hermitian metrics}

In this section, we consider the case where a Hermitian metric of a holomorphic vector bundle has singularities
and investigate its approximation and properties.

\subsection{Definition of positivity}

First, we introduce singular Hermitian metrics on holomorphic line bundles and define its positivity.

\begin{definition}$(\mathrm{cf.~[5],\,[7,\,Chapter\,3]})$ 
    A $\it{singular}$ $\it{Hermitian}$ $\it{metric}$ $h$ on a line bundle $L$ is a metric which is given in any trivialization $\tau:L|_U\xrightarrow{\simeq} U\times\mathbb{C}$ by 
    \begin{align*}
        ||\xi||_h=|\tau(\xi)|e^{-\varphi}, \qquad x\in U,\,\,\xi\in L_x
    \end{align*}
    where $\varphi\in\mathcal{L}^1_{loc}(U)$ is an arbitrary function, called the weight of the metric with respect to the trivialization $\tau$.
\end{definition}

\begin{definition}\label{def of psef & big}
    Let $L$ be a holomorphic line bundle on a complex manifold $X$ equipped with a singular Hermitian metric $h$. 
    \begin{itemize}
        \item [($a$)] $h$ is $\it{singular}$ $\it{semi}$-$\it{positive}$ if $i\Theta_{L,h}\geq0$ in the sense of currents,
        i.e. the weight of $h$ with respect to any trivialization coincides with some plurisubharmonic function almost everywhere.
        \item [($b$)] $h$ is $\it{singular}$ $\it{positive}$ if the weight of $h$ with respect to any trivialization coincides with some strictly plurisubharmonic function almost everywhere.
        \item [($c$)] Let $\omega$ be a \kah metric on $X$. Then $h$ is $\it{strictly}$ $\delta_\omega$-$\it{positive}$ if for any open subset $U$ and any \kah potential $\varphi$ of $\omega$ on $U$, $he^{\delta\varphi}$ is singular semi-positive.
    \end{itemize}
\end{definition}

Clearly, singular semi-positivity is coincides with pseudo-effective on compact complex manifolds. 
Furthermore, singular positivity and strictly $\delta_\omega$-positivity also coincide with big on compact \kah manifolds by Demailly's definition and characterization (see \cite{Dem93},\,[7,\,Chapter\,6]), where $\omega$ is a \kah metric. 

The Lelong number of a plurisubharmonic function $\varphi$ on $X$ is defined by 
\begin{align*}
    \nu(\varphi,x):=\liminf_{z\to x}\frac{\varphi(z)}{\log|z-x|}
\end{align*}
for some coordinate $(z_1,\cdots,z_n)$ around $x\in X$.
For the relationship between the Lelong number of $\varphi$ and the integrability of $e^{-\varphi}$, the following important result obtained by Skoda in \cite{Sko72} is known; If $\nu(\varphi,x)<1$ then $e^{-2\varphi}$ is integrable around $x$.
From this, particularly if $\nu(-\log h,x)<2$ then $\mathscr{I}(h)=\mathcal{O}_{X,x}$ immediately.

For holomorphic vector bundles, we introduce the definition of singular Hermitian metrics $h$ and 
the $L^2$-subsheaf $\mathscr{E}(h)$ of $\mathscr{O}(E)$ analogous to the multiplier ideal sheaf.

\begin{definition}$(\mathrm{cf.~[2,~Section~3],~~[31,~Definition,~2.2.1]~and~[33,~Definition~1.1]})$ 
    We say that $h$ is a $\it{singular~Hermitian~metric}$ on $E$ if $h$ is a measurable map from the base manifold $X$ to the space of non-negative Hermitian forms on the fibers satisfying $0<\mathrm{det}\,h<+\infty$ almost everywhere.
\end{definition}

\begin{definition}\label{def of multiplier submodule}$(\mathrm{cf.\,[3,\,Definition\,2.3.1]})$ 
    Let $h$ be a singular Hermitian metric on $E$. We define the $L^2$-subsheaf $\mathscr{E}(h)$ of germs of local holomorphic sections of $E$ as follows:
    \begin{align*}
        \mathscr{E}(h)_x:=\{s_x\in\mathscr{O}(E)_x\mid|s_x|^2_h~ \mathrm{is ~locally ~integrable ~around~} x\}.
    \end{align*}
\end{definition}

Moreover, we introduce the definitions of positivity and negativity, such as Griffiths and Nakano, for singular Hermitian metrics.

\begin{definition}\label{def Griffiths semi-posi sing}$(\mathrm{cf.~[4,~Definition~3.1],~[34,~Definition~2.2.2]~and~[36,~Definition~1.2]})$ 
    We say that a singular Hermitian metric $h$ is 
    \begin{itemize}
        \item [(1)] $\textit{Griffiths semi-negative}$ if $||u||_h$ is plurisubharmonic for any local holomorphic section $u\in\mathscr{O}(E)$ of $E$.
        \item [(2)] $\textit{Griffiths semi-positive}$ if the dual metric $h^*$ on $E^*$ is Griffiths semi-negative.
    \end{itemize}
\end{definition}

Let $h$ be a smooth Hermitian metric on $E$ and $u=(u_1,\cdots,u_n)$ be an $n$-tuple of locally holomorphic sections of $E$. We define $T^h_u$, an $(n-1,n-1)$-form through
\begin{align*}
    T^h_u=\sum_{1\leq j,k\leq n}(u_j,u_k)_h\widehat{dz_j\wedge d\overline{z}_k}
\end{align*}
where $(z_1,\cdots,z_n)$ are local coordinates on $X$ and $\widehat{dz_j\wedge d\overline{z}_k}$ 
satisfying $idz_j\wedge d\overline{z}_k\wedge\widehat{dz_j\wedge d\overline{z}_k}=dV_{\mathbb{C}^n}$.
Then a short computation yields that $(E,h)$ is Nakano semi-negative if and only if $T^h_u$ is plurisubharmonic in the sense that $\idd T^h_u\geq0$ (see [1,\,33]). 
In the case of $u_j=\xi_ju$ for any $\xi\in\mathbb{C}^n$, 
$(E,h)$ is Griffiths semi-negative.

Let $h$ be a singular Hermitian metric of $E$. For any $n$-tuple of locally holomorphic sections $u=(u_1,\cdots,u_n)$, we say that the $(n-1,n-1)$-form $T^h_u$ is plurisubharmonic if $\idd T^h_u\geq0$ in the sense of currents.        
From the above, we introduce the definitions of Nakano semi-negativity and dual Nakano semi-positivity for singular Hermitian metrics.

\begin{definition}\label{def Nakano semi-negative as Raufi}$\mathrm{(cf.\,[33,\,Section\,1]})$ 
    We say that a singular Hermitian metric $h$ on $E$ is $\it{Nakano}$ $\it{semi}$-$\it{negative}$ if the $(n-1,n-1)$-form $T^h_u$ is plurisubharmonic for any $n$-tuple of locally holomorphic sections $u=(u_1,\cdots,u_n)$.
\end{definition}

\begin{definition}\label{def dual Nakano semi-posi sing}$\mathrm{(cf.\,[39,\,Definition\,4.5]})$ 
    We say that a singular Hermitian metric $h$ on $E$ is $\it{dual}$ $\it{Nakano}$ $\it{semi}$-$\it{positive}$ if the dual metric $h^*$ on $E^*$ is Nakano semi-negative.
\end{definition}

Since the dual of a Nakano negative bundle in general is not Nakano positive, we cannot define Nakano semi-positivity for singular Hermitian metrics as in the case of Griffiths, 
but this definition of dual Nakano semi-positivity is natural.
We already know one definition of Nakano semi-positivity for singular Hermitian metrics in \cite{Ina22} as follows, which is based on the optimal $L^2$-estimate condition in [11,\,16] 
and is equivalent to the usual definition for the smooth case.

\begin{definition}\label{def Nakano semi-posi sing by Inayama}$\mathrm{(cf.\,[19,\,Definition\,1.1]})$ 
    Assume that $h$ is a Griffiths semi-positive singular Hermitian metric. We say that $h$ is 
    \textit{Nakano semi-positive} if for any Stein coordinate $S$ such that $E|_S$ is trivial, any \kah metric $\omega_S$ on $S$,
    any smooth strictly plurisubharmonic function $\psi$ on $S$, any positive integer $q\in\{1,\cdots,n\}$ and any $\overline{\partial}$-closed $f\in L^2_{n,q}(S,E,he^{-\psi},\omega_S)$ 
    there exists $u\in L^2_{n,q-1}(S,E,he^{-\psi},\omega_S)$ satisfying $\overline{\partial}u=f$ and 
    \begin{align*}
        \int_S|u|^2_{h,\omega_S}e^{-\psi}dV_{\omega_S}\leq\int_S\langle B^{-1}_{\psi,\omega_S}f,f\rangle_{h,\omega_S}e^{-\psi}dV_{\omega_S},
    \end{align*}
    where $B_{\psi,\omega_S}=[\idd\psi\otimes\mathrm{id}_E,\Lambda_{\omega_S}]$. Here we assume that the right-hand side is finite.
\end{definition}

However, this definition loses the invariance of the positivity in tensor products, i.e. the tensor product of Nakano semi-positive and $m$-semi-positive is also $m$-semi-positive.
By using Proposition \ref{characterization of smooth Nakano}, we can introduce the definition including this invariance.

\begin{definition}\label{def Nakano semi-posi sing}
    Assume that $h$ is a Griffiths semi-positive singular Hermitian metric. We say that $h$ is 
    $L^2$-$\it{type}$ $\it{Nakano}$ $\it{semi}$-$\it{positive}$ if for any positive integer $k\in\{1,\cdots,n\}$, any Stein coordinate $S$, any \kah metric $\omega_S$ on $S$ and
    any smooth Hermitian metric $h_F$ on any holomorphic vector bundle $F$ such that $A^{n,s}_{F,h_F,\omega_S}>0$ for $s\geq k$,
    we have that any positive integer $q\geq k$ and any $\overline{\partial}$-closed $f\in L^2_{n,q}(S,E\otimes F,h\otimes h_F,\omega_S)$ 
    there exists $u\in L^2_{n,q-1}(S,E\otimes F,h\otimes h_F,\omega_S)$ satisfying $\overline{\partial}u=f$ and 
    \begin{align*}
        \int_S|u|^2_{h\otimes h_F,\omega_S}dV_{\omega_S}\leq\int_S\langle B^{-1}_{h_F,\omega_S}f,f\rangle_{h\otimes h_F,\omega_S}dV_{\omega_S},
    \end{align*}
    where $B_{h_F,\omega_S}=[i\Theta_{F,h_F}\otimes\mathrm{id}_E,\Lambda_{\omega_S}]$. Here we assume that the right-hand side is finite.
\end{definition}

Notes that from Lemma \ref{trivial of E on Stein-H} and \ref{Ext d-equation for hypersurface}, the assumption of triviality for the vector bundle in Definition \ref{def Nakano semi-posi sing by Inayama} can be excluded.
Furthermore, obviously $h$ is Nakano semi-positive in the sense of Definition \ref{def Nakano semi-posi sing by Inayama} if it is $L^2$-type Naknao semi-positive, and the above definitions \ref{def Griffiths semi-posi sing}-\ref{def Nakano semi-posi sing} does not require the use of curvature currents.
For singular Hermitian metrics we cannot always define the curvature currents with measure coefficients \cite{Rau15}. 

In \cite{Nad89}, Nadel proved that $\mathscr{I}(h)$ is coherent by using the H\"ormander $L^2$-estimate. 
After that, as vector bundles case, Hosono and Inayama proved that $\mathscr{E}(h)$ is coherent if $h$ has Nakano semi-positivity in the sense of Definition \ref{def Nakano semi-posi sing by Inayama} (or \ref{def Nakano semi-posi sing}) in [16,\,19]. 

Finally we introduce the strictly positivity for Griffiths and Nakano is known.

\begin{definition}\label{def sing strictly positive of Grif and (dual) Nakano}$\mathrm{(cf.\,[18,\,Definition\,2.6],\,[19,\,Definition\,2.16],\,[39,\,Definition\,4.11])}$  
    Let $(X,\omega)$ be a \kah manifold and $h$ be a singular Hermitian metric on $E$.
    \begin{itemize}
        \item We say that $h$ is $\textit{strictly Griffiths}$ $\delta_\omega$-$\textit{positive}$ if for any open subset $U$ and any \kah potential $\varphi$ of $\omega$ on $U$, $he^{\delta\varphi}$ is Griffiths semi-positive on $U$.
        \item We say that $h$ is $L^2$-$\it{type}$ $\textit{strictly Nakano}$ $\delta_\omega$-$\textit{positive}$ if for any open subset $U$ and any \kah potential $\varphi$ of $\omega$ on $U$, $he^{\delta\varphi}$ is $L^2$-type Nakano semi-positive on $U$ in the sense of Definition \ref{def Nakano semi-posi sing}.
        \item We say that $h$ is $\textit{strictly dual Nakano}$ $\delta_\omega$-$\textit{positive}$ if for any open subset $U$ and any \kah potential $\varphi$ of $\omega$ on $U$, $he^{\delta\varphi}$ is dual Nakano semi-positive on $U$.
    \end{itemize}
\end{definition}

On projective manifolds, it is known that these strictly $\delta_\omega$-positivity for Griffiths and (dual) Nakano give $L^2$-estimates (see. [19,\,Theorem\,1.4], [39,\,Theorem\,4.12]) 
and induce vanishing theorems involving $L^2$-subsheaves (see. [19,\,Theorem\,1.5], [39,\,Theorem\,1.3]). 
In this paper, we consider the $L^2$-estimates and vanishing theorems for singular type (dual) Nakano semi-positive Hermitian metrics twisted by smooth (dual) $m$-positive Hermitian metrics as more natural positivity on weakly pseudoconvex \kah manifolds.

\subsection{Approximation and properties of singular Hermite metrics}

For singular semi-positivity on line bundles, the following Demailly's approximation is known.

\begin{theorem}\label{Demailly smoothing of current}$\mathrm{(cf.\,[6,\,Theorem\,6.1]})$ 
    Let $(X,\omega)$ be a complex manifold equipped with a Hermitian metric $\omega$ and $\Omega\!\Subset \!X$ be an open subset.
    Assume that $T=\alpha+\frac{i}{\pi}\partial\overline{\partial}\varphi$ is a closed $(1,1)$-current on $X$, where $\alpha$ is a smooth real $(1,1)$-form in the same $\partial\overline{\partial}$-cohomology class as $T$ and $\varphi$ is a quasi-plurisubharmonic function.
    Let $\gamma$ be a continuous real $(1,1)$-form such that $T\geq\gamma$.
    Suppose that the Chern curvature tensor of $T_X$ satisfies
    \begin{align*}
        (i\Theta_{T_X}+\theta\otimes\mathrm{id}_{T_X})(\kappa_1\otimes\kappa_2,\kappa_1\otimes\kappa_2)\geq0 \quad \forall\kappa_1,\kappa_2\in T_X~\mathrm{with}~\langle\kappa_1,\kappa_2\rangle=0
    \end{align*}
    on a neighborhood of $\overline{\Omega}$, for some continuous nonnegative $(1,1)$-form $\theta$ on $X$. Then for every $c>0$, there is a family of closed $(1,1)$-currents $T_{c,\varepsilon}=\alpha+\frac{i}{\pi}\partial\overline{\partial}\varphi_{c,\varepsilon}$ such that 
    \begin{itemize}
        \item [(i)] $\varphi_{c,\varepsilon}$ is quasi-plurisubharmonic on a neighborhood of $\overline{\Omega}$, smooth on $X\setminus E_c(T)$, increasing with respect to $\varepsilon$ on $\Omega$, and converges to $\varphi$ on $\Omega$ as $\varepsilon\to0$,
        \item [(ii)] $T_{c,\varepsilon}\geq\gamma-c\theta-\delta_\varepsilon\omega$ on $\Omega$,
    \end{itemize}
    where $\varepsilon\in(0,\varepsilon_0)$, $E_c(T)=\{x\in X\mid\nu(T,x)\geq c\}$ is the $c$-upperlevel set of Lelong numbers and $(\delta_\varepsilon)_{\varepsilon>0}$ is an increasing family of positive numbers such that $\lim_{\varepsilon\to0}\delta_\varepsilon=0$.
\end{theorem}

\begin{remark}$\mathrm{(cf.\,[41,\,Remark\,3.1]})$ 
    Although Theorem \ref{Demailly smoothing of current} is stated in \cite{Dem94} when $X$ is compact, almost the same proof as in \cite{Dem94} shows that Theorem \ref{Demailly smoothing of current} holds in the noncompact case while uniform estimates are obtained only on the relatively compact subset. 
\end{remark}

We consider the approximation of singular Hermitian metrics using convolution by the mollifier.
Let $S$ be a Stein manifold. We may assume that $S$ is a submanifold of $\mathbb{C}^{N}$.  
By the theorem of Docquier and Grauert, there exists an open neighborhood $W \subset \mathbb{C}^{N}$ of $S$ and a holomorphic retraction $\mu: W \to S$ (cf. ChapterV\,of\,\cite{Hor90}).
Let $\rho:\mathbb{C}^N\rightarrow \mathbb{R}_{\geq0}$ be a smooth function depending only on $|z|$ such that $\mathrm{supp}\,\rho\subset \mathbb{B}^N$ and that $\int_{\mathbb{C}^N} \rho(z)dV=1$, where $\mathbb{B}^N$ is the unit ball. 
Define the mollifier $\rho_\nu(z)=\nu^{2N}\rho(\nu z)$ for $\nu>0$, where $\rho_\nu\to \delta$ delta distribution if $\nu\to+\infty$. 
For any subset $D\subset\mathbb{C}^N$, let $D^\nu:=\{z\in D\mid d_N(z,\partial D)>1/\nu\}\subset\subset D$.
Then for any function $\varphi$ over $D$, the convolution $\varphi_\nu:=\varphi\ast\rho_\nu$ is a smooth function defined on $D^\nu$.

Here, for any Stein manifold $S$, we say that the mollifier sequence $(\rho_\nu)_{\nu\in\mathbb{N}}$ is an $\it{approximate}$ $\it{identity}$ $\it{with}$ $\it{respect}$ $\it{to}$ $S$.

The following are known for approximations of singular Hermitian metrics on holomorphic vector bundles.

\begin{proposition}\label{vect bdl semi-negative sHm smoothing}$\mathrm{(cf.\,[2,\,Proposition\,3.1],\,[39,\,Proposition\,4.10]})$ 
    Let $S$ be a Stein manifold and $E$ be a holomorphic vector bundle over $S$ equipped with a singular Hermitian metric $h$. 
    We assume that $E$ is trivial over $S$. Then we have the following
    \begin{itemize}
        \item [($a$)] $h$ is Griffiths semi-negative if and only if there exists a sequence of smooth Griffiths semi-negative Hermitian metrics $(h_\nu)_{\nu\in\mathbb{N}}$ decreasing to $h$ $a.e.$ pointwise on any relatively compact Stain subset of $S$.
        \item [($b$)] If $h$ is Nakano semi-negative then there exists a sequence of smooth Nakano semi-negative Hermitian metrics $(h_\nu)_{\nu\in\mathbb{N}}$ decreasing to $h$ $a.e.$ pointwise on any relatively compact Stain subset of $S$, where $h_\nu=h\ast\rho_\nu$.
    \end{itemize}

    Here, we can always construct smooth Hermite metrics $h_\nu$ on $E$ over any relatively compact Stain open subset of $S$ 
    by convolving the function $\rho_\nu$, i.e. $h_\nu=h\ast\rho_\nu$, where $(\rho_\nu)_{\nu\in\mathbb{N}}$ is an approximate identity with respect to $S$.
\end{proposition}

The following lemma shows that the difference between the above assumption of triviality for vector bundles and not assuming it is only about a hypersurface.

\begin{lemma}\label{trivial of E on Stein-H}
    Let $S$ be a Stein manifold and $E$ be a holomorphic vector bundle on $S$. Then there exists a hypersurface $H$ such that $E|_{S\setminus H}$ is trivial, where $S\setminus H$ is also Stein.
\end{lemma}

\begin{proof}
    Let $r$ be a rank of $E$, $f=(\sigma_1,\cdots,\sigma_r)$ be a $r$-tuple of globally holomorphic sections of $E$, i.e. $\sigma_j\in H^0(S,E)$, for $1\leq \alpha\leq n+1$. We define the hypersurface by $H:=\{z\in S\mid \Lambda^r_{j=1}\sigma_j(z)=0\}$, 
    where $\Lambda^r_{j=1}\sigma_j\in H^0(S,\mathrm{det}\,E)$. Here, $S\setminus H$ is also Stein (see \cite{Die96}).
    We define the holomorphic map $\tau:S\setminus H\times \mathbb{C}^r\to E|_{S\setminus H}$ by $\tau(z,\xi)=f(z)\cdot\xi=\sum^r_{j=1}\xi_j\sigma_j(z)$ where $\xi=\!\,^t(\xi_1,\cdots,\xi_r)$, 
    then it is holomorphic isomorphism by $f$ is globally holomorphic frame on $S\setminus H$. Hence $E|_{S\setminus H}$ is trivial.
\end{proof}

We propose one effective method of determining singular Nakano semi-negativity.

\begin{proposition}\label{key Prop of sing Nakano semi-negative}
    Let $S$ be a Stein manifold and $E$ be a holomorphic vector bundle which is trivial over $S$ equipped with a singular Hermitian metric $h$. 
    If $(h_\nu)_{\nu\in\mathbb{N}}$ is a sequence of smooth Nakano semi-negative Hermitian metrics  
    then $h$ is Nakano semi-negative, where $h_\nu=\!h\ast\rho_\nu$ and $(\rho_\nu)_{\nu\in\mathbb{N}}$ is an approximate identity with respect to $S$.
\end{proposition}

\begin{proof}
    First, we show Griffiths semi-negativity of $h$. It is sufficient to show that $(h_\nu)_{\nu\in\mathbb{N}}$ decreases to $h$ $a.e.$ pointwise.
    By smooth Griffiths semi-negativity of $h_\delta$, for any locally constant section $s\in\mathcal{O}_S(E)$, the smooth function $||s||^2_{h_\delta}=||s||^2_h\ast\rho_\delta$ is plurisubharmonic. 
    
    For any positive integers $\nu>\mu$, we have that 
    \begin{align*}
        ||s||^2_{h_\delta}\ast\rho_\nu\geq||s||^2_{h_\delta}\ast\rho_\mu \quad \mathrm{and} \quad ||s||^2_{h_\delta}\ast\rho_\nu=||s||^2_h\ast\rho_\delta\ast\rho_\nu=||s||^2_{h_\nu}\ast\rho_\delta.
    \end{align*}
    Therefore, we obtain $||s||^2_{h_\nu}\geq||s||^2_{h_\mu}$ by taking the limit of $||s||^2_{h_\nu}\ast\rho_\delta\geq||s||^2_{h_\mu}\ast\rho_\delta$ as $\delta\to+\infty$. Hence, $(h_\nu)_{\nu\in\mathbb{N}}$ is decreasing and converges to $h$ a.e. pointwise.

    For any fixed point $x_0\in S$, there exist a open neighborhood $U$ of $x_0$ and $\nu_0\in\mathbb{N}$ such that $U\subset S^{\nu_0}\subset S^\nu$ for any $\nu\geq\nu_0$.
    For any $n$-tuple locally holomorphic sections $u=(u_1,\cdots,u_n)$ of $E$, i.e. $u_j\in H^0(U,E)$, we have that 
    \begin{align*}
        (u_j,u_k)_{h_\nu}(z)=\int (u_j,u_k)_{h^{(w)}}(z)\rho_\nu(w)dV_w, \quad
        T^{h_\nu}_u(z)=\int T^{h^{(w)}}_u(z)\rho_\nu(w)dV_w,
    \end{align*} 
    where $h^{(w)}(z)=h(z-w)$.
    For any nonnegative test form $\phi\in\mathscr{D}(U)$ we have that 
    \begin{align*}
        0\leq\idd T^{h_\nu}_u(\phi)&=\int\phi\idd T^{h_\nu}_u=\int T^{h_\nu}_u\wedge\idd\phi\\
        &=\int_z\Bigl\{\int_w T^{h^{(w)}}_u(z)\rho_\nu(w)dV_w\Bigr\}\wedge\idd\phi\\
        &=\int_w\Bigl\{\int_zT^{h^{(w)}}_u(z)\wedge\idd\phi\Bigr\}\rho_\nu(w)dV_w  \\  
        &=\int_{w\in\mathrm{supp}\,\rho_\nu}\idd T^{h^{(w)}}_u(\phi)\rho_\nu(w)dV_w.
    \end{align*}

    Define the function $F=\idd T^{h^{(\bullet)}}_u(\phi):\mathrm{int}(\mathrm{supp}\,\rho_\nu)\to\mathbb{R}$. 
    We show $F(0)=\idd T^h_u(\phi)\geq0$, i.e. Nakano semi-negativity of $h$. 
    For any $\zeta\in\mathbb{C}^n$ enough close to $0$, we obtain
    \begin{align*}
        F(\zeta)&=\idd T^{h^{(\zeta)}}_u(\phi)=\int_{w\in U}T^{h^{(\zeta)}}_u(w)\wedge\idd\phi(w)\\
        &=\int_U\sum (u_j(w),u_k(w))_{h^{(\zeta)}(w)}\widehat{dz_j\wedge d\overline{z}_k}\wedge\idd\phi(w)\\
        &=\int_U\sum (u_j(w),u_k(w))_{h(w-\zeta)}\phi_{jk}(w)dV_w\\
        &=\int_{U-\zeta}\sum (u_j(\xi+\zeta),u_k(\xi+\zeta))_{h(\xi)}\phi_{jk}(\xi+\zeta)dV_\xi\\
        &=\int_U\sum (u_j(w+\zeta),u_k(w+\zeta))_{h(w)}\phi_{jk}(w+\zeta)dV_w
    \end{align*}
    where $\phi_{jk}=\frac{\partial^2\phi}{\partial z_j\partial\overline{z}_k}$, $\xi=w-\zeta$ 
    and we take a enough small $\zeta$ satisfying $\mathrm{supp}\,\phi+\zeta\subset U$. 

    For any $\zeta\in\mathrm{int}(\mathrm{supp}\,\rho_\nu)$ enough close to $0$ and any $w\in U$, we define the function
    \begin{align*}
        g(\zeta,w)=\sum (u_j(w+\zeta),u_k(w+\zeta))_{h(w)}\phi_{jk}(w+\zeta) \quad \mathrm{then} \quad F(\zeta)=\int_Ug(\zeta,w)dV_w.
    \end{align*}

    Here, from Griffiths semi-negativity of $h$, each element $h_{jk}$ is bounded (see [31,\,Lemma\,2.2.4]). 
    Therefore, there exists a integrable function $M:U\to\mathbb{R}_{\geq0}$ such that $|g(\zeta,w)|\leq M(w)$ for any $w\in U$ and any $\zeta$ enough close to $0$.
    Since Lebesgue's dominated convergence theorem, for any $\zeta_0$ enough close to $0$ we have that 
    \begin{align*}
        \lim_{\zeta\to+0}F(\zeta_0+\zeta)&=\lim_{\zeta\to+0}\int_Ug(\zeta_0+\zeta,w)dV_w\\
        &=\int_U\lim_{\zeta\to+0}g(\zeta_0+\zeta,w)dV_w=\int_Ug(\zeta_0,w)dV_w=F(\zeta_0),
    \end{align*}
    where for any $w\in U$, $g(\zeta,w)$ is smooth as to $\zeta$ by smoothness of $u_j$ and $\phi_{jk}$.
    Therefore $F$ is continuous near $0$. 
    Since smooth Nakano semi-negativity of $h_\nu$, we obtain that 
    \begin{align*}
        0\leq\lim_{\nu_0\leq\nu\to+\infty}\idd T^{h_\nu}_u(\phi) 
        &=\lim_{\nu\to+\infty}\int_wF(w)\rho_\nu(w)dV_w\\
        &=\lim_{\nu\to+\infty}\langle\rho_\nu,F\rangle
        =\langle\delta_0,F\rangle
        =F(0).
    \end{align*}
    Hence, $h$ is Nakano semi-negative.
\end{proof}

We obtain the following basic properties for semi-positivity in the sense of Griffiths, $L^2$-type Nakano and dual Nakano with respect to tensor products of line bundles.

\begin{theorem}\label{psef * Nak semi-posi then Nak semi-posi}
    Let $X$ be a complex manifold, $L$ be a holomorphic line bundle over $X$ equipped with a singular Hermitian metric $h_L$ 
    and $E$ be a holomorphic vector bundle over $X$ equipped with a singular Hermitian metric $h_E$. We have the following
    \begin{itemize}
        \item [($a$)] If $h_L$ is singular semi-positive and $h_E$ is Griffiths semi-positive then $h_E\otimes h_L$ is also Griffiths semi-positive.
        \item [($b$)] If $h_L$ is singular semi-positive and $h_E$ is $L^2$-type Nakano semi-positive then $h_E\otimes h_L$ is also $L^2$-type Nakano semi-positive.
        \item [($c$)] If $h_L$ is singular semi-positive and $h_E$ is dual Nakano semi-positive then $h_E\otimes h_L$ is also dual Nakano semi-positive.
    \end{itemize}
\end{theorem}

\begin{proof}
    $(a)$ We already know that $h_E$ is Griffiths semi-positive if and only if $\log|u|_{h_E^*}$ is plurisubharmonic, i.e. $\idd\log|u|^2_{h_E^*}\geq0$ in the sense of currents, for any local holomorphic section $u\in\mathcal{O}(E^*)$ of $E^*$.
    Let $U$ be any open subset such that $L$ is trivial on $U$.
    Hence, for any local holomorphic section $u\in\mathcal{O}(E^*\otimes L^*)(U)=\mathcal{O}(E^*)(U)$, we have that $\idd\log|u|^2_{h_E^*\otimes h_L^*}=\idd\log|u|^2_{h_E^*}-\idd\log h_L\geq i\Theta_{L,h_L}\geq0$ in the sense of currents.

    $(b)$ We fix any positive integer $k\in\{1,\cdots,n\}$, any Stein coordinate $S$, any \kah metric $\omega_S$ on $S$ and any smooth Hermitian metric $h_F$ on any holomorphic vector bundle $F$ such that $A^{n,s}_{F,h_F,\omega_S}>0$ for $s\geq k$.
    
    By Lemma \ref{trivial of E on Stein-H}, there is a hypersurface $H$ such that $S_H\!:=\!S\!\setminus\!H$ is also Stein and $L|_{S_H}$ is trivial.
    There is a strictly plurisubharmonic function $\psi$ on $S_H$ which is smooth exhaustive and $\sup_{S_H}\psi=\!+\infty$.  
    Let $S_H(j):=\{z\in S_H\mid \psi(z)<j\}$ be Stein sublevel sets.
    Fixed $j\in\mathbb{N}$. There is $\nu_1\in\mathbb{N}$ such that for any integer $\nu\geq\nu_1$, $S_H(j)\Subset S^{\nu_1}_H \Subset S^\nu_H$. 
    
    By Proposition \ref{vect bdl semi-negative sHm smoothing}, there is a sequence of smooth semi-positive metrics $(h_\nu)_{\nu\in\mathbb{N}}$ increasing to $h_L$, 
    where $h_\nu:=(h_L^*\ast\rho_\nu)^*$ defined on $S^\nu_H$. 
    For any $\nu\in\mathbb{N}$, we have that $A^{n,t}_{L,h_\nu,\omega_S}\geq0$ for $t\geq1$ 
    and $A^{n,s}_{F\otimes L,h_F\otimes h_\nu,\omega_S}\geq0$ for $s\geq k$ by Proposition \ref{A_(E+F)>C_E+C_F if A_E>C_E and A_F>C_F} and that 
    \begin{align*}
        B_{h_F\otimes h_\nu,\omega_S}
        &=[i\Theta_{F,h_F}\otimes\mathrm{id}_{L\otimes E},\Lambda_{\omega_S}]+[i\Theta_{L,h_\nu}\otimes\mathrm{id}_{F\otimes E},\Lambda_{\omega_S}]\\
        &\geq [i\Theta_{F,h_F}\otimes\mathrm{id}_{L\otimes E},\Lambda_{\omega_S}]=B_{h_F,\omega_S},
    \end{align*}
    where $B_{h_F,\omega_S}>0$ for any $(n,q)$-forms with $q\geq k$.

    Here, for any positive integers $q\geq k$ and any $\overline{\partial}$-closed $f\in L^2_{n,q}(S,E\otimes L\otimes F,h_E\otimes h_L\otimes h_F,\omega_S)$ such that $\int_S\langle B^{-1}_{h_F,\omega_S}f,f\rangle_{h_E\otimes h_L\otimes h_F,\omega_S}dV_{\omega_S}<+\infty$,
    we have that 
    \begin{align*}
        \int_S\langle B^{-1}_{h_F,\omega_S}f,f\rangle_{h_E\otimes h_L\otimes h_F,\omega_S}dV_{\omega_S}
        &\geq\int_{S_H(j)}\langle B^{-1}_{h_F,\omega_S}f,f\rangle_{h_E\otimes h_\nu\otimes h_F,\omega_S}dV_{\omega_S}\\
        &\geq\int_{S_H(j)}\langle B^{-1}_{h_F\otimes h_\nu,\omega_S}f,f\rangle_{h_E\otimes h_\nu\otimes h_F,\omega_S}dV_{\omega_S}.
    \end{align*}

    By $L^2$-type Nakano semi-positivity of $h_E$, there exists $u_{j,\nu}\in L^2_{n,q-1}(S_H(j),E\otimes L\otimes F,h_E\otimes h_\nu\otimes h_F,\omega_S)$ such that $\overline{\partial}u_{j,\nu}=f$ on $S_H(j)$ and
    \begin{align*}
        \int_{S_H(j)}|u_{j,\nu}|^2_{h_E\otimes h_\nu\otimes h_F,\omega_S}dV_{\omega_S}&\leq\int_{S_H(j)}\langle B^{-1}_{h_F\otimes h_\nu,\omega_S}f,f\rangle_{h_E\otimes h_\nu\otimes h_F,\omega_S}dV_{\omega_S}\\
        &\leq\int_S\langle B^{-1}_{h_F,\omega_S}f,f\rangle_{h_E\otimes h_L\otimes h_F,\omega_S}dV_{\omega_S}<+\infty.
    \end{align*}

    Since the monotonicity to $\nu$ of $|\bullet|^2_{h_E\otimes h_\nu\otimes h_F,\omega}$ by the increase in $(h_\nu)_{\nu\in\mathbb{N}}$, the sequence
    $(u_{j,\nu})_{\nu_1\leq\nu\in\mathbb{N}}$ forms a bounded sequence in $L^2_{n,q-1}(S_H(j),E\otimes L\otimes F,h_E\otimes h_{\nu_1}\otimes h_F,\omega_S)$.
    Thus, we can obtain a weakly convergence subsequence in $L^2_{n,q-1}(S_H(j),E\otimes L\otimes F,h_E\otimes h_{\nu_1}\otimes h_F,\omega_S)$. 
    By using a diagonal argument, we get a subsequence $(u_{j,\nu_k})_{k\in\mathbb{N}}$ of $(u_{j,\nu})_{\nu\in\mathbb{N}}$ converging weakly in $L^2_{n,q-1}(S_H(j),E\otimes L\otimes F,h_E\otimes h_{\nu_1}\otimes h_F,\omega_S)$ for any $\nu_1$, where
    $u_{j,\nu_k}\in L^2_{n,q-1}(S_H(j),E\otimes L\otimes F,h_E\otimes h_{\nu_k}\otimes h_F,\omega_S)\subset L^2_{n,q-1}(S_H(j),E\otimes L\otimes F,h_E\otimes h_{\nu_1}\otimes h_F,\omega_S)$.

    Denote $u_j$ by the weakly limit of $(u_{j,\nu_k})_{k\in\mathbb{N}}$. Then $u_j$ satisfies $\overline{\partial}u_j=f$ on $S_H(j)$ and 
    \begin{align*}
        \int_{S_H(j)}|u_j|^2_{h_E\otimes h_{\nu_k}\otimes h_F,\omega_S}\leq\int_S\langle B^{-1}_{h_F,\omega_S}f,f\rangle_{h_E\otimes h_L\otimes h_F,\omega_S}dV_{\omega_S}<+\infty,
    \end{align*}
    for any $k\in\mathbb{N}$. Taking weakly limit $k\to+\infty$ and using the monotone convergence theorem, we have the following estimate
    \begin{align*}
        \int_{S_H(j)}|u_j|^2_{h_E\otimes h_L\otimes h_F,\omega_S}\leq\int_S\langle B^{-1}_{h_F,\omega_S}f,f\rangle_{h_E\otimes h_L\otimes h_F,\omega_S}dV_{\omega_S}<+\infty.
    \end{align*}

    Here, let $\chi_j\in\mathscr{D}(S_H(j),\mathbb{R}_{\geq0})$ be a cut-off function satisfies $\chi_j\equiv1$ on $S_H(j-1)$, $\mathrm{supp}\,\chi_j\subset\subset S_H(j)$ and $0\leq\chi_j\leq1$.
    We define $v_j:=\chi_ju_j\in L^2_{n,q-1}(S_H,E\otimes L\otimes F,h_E\otimes h_L\otimes h_F,\omega_S)$ then $v_j$ satisfies $\overline{\partial}v_j=f$ on $S_H(j-1)$ and 
    \begin{align*}
        \int_{S_H(j)}|v_j|^2_{h_E\otimes h_L\otimes h_F,\omega_S}\leq\int_{S_H(j)}|u_j|^2_{h_E\otimes h_L\otimes h_F,\omega_S}\leq\int_S\langle B^{-1}_{h_F,\omega_S}f,f\rangle_{h_E\otimes h_L\otimes h_F,\omega_S}dV_{\omega_S}<+\infty.
    \end{align*}

    Repeating the above argument and taking weak limit $j\to+\infty$, we get the solution $v \in L^2_{n,q-1}(S_H,E\otimes L\otimes F,h_E\otimes h_L\otimes h_F,\omega_S)$ of $\overline{\partial}v=f$ on $S_H$ such that 
    \begin{align*}
        \int_{S_H}|v|^2_{h_E\otimes h_L\otimes h_F,\omega_S}dV_{\omega_S}=\int_S|v|^2_{h_E\otimes h_L\otimes h_F,\omega_S}dV_{\omega_S}\leq\int_S\langle B^{-1}_{h_F,\omega_S}f,f\rangle_{h_E\otimes h_L\otimes h_F,\omega_S}dV_{\omega_S},
    \end{align*}
    where Lebesgue measure of $H$ is zero.
    By Lemma \ref{Ext d-equation for hypersurface} in $\S4.1$, letting $v=0$ on $H$ then we have that $\overline{\partial}v=f$ on $S$.
    Hence, 
    $h_E\otimes h_L$ is also $L^2$-type Nakano semi-positive.

    $(c)$ For simplicity, let $h_L^*$ be singular semi-positive and $h_E$ be Nakano semi-negative. We prove that $h_E\otimes h_L$ is Nakano semi-negative.
    Since Proposition \ref{key Prop of sing Nakano semi-negative} and Nakano semi-negativity is locally property, it is sufficient to show that $(h_E\otimes h_L)\ast\rho_\nu$ is Nakano semi-negative on any locally open subset for each $\nu\in\mathbb{N}$.

    First, for a smooth semi-negative Hermitian metric $h$ on $L$, we show that $h\otimes h_E$ is Nakano semi-negative.
    For any fixed point $x_0\in X$, there exists an open Stein neighborhood $U$ of $x_0$ such that $E|_U$ and $L|_U$ are trivial. Let $h_E^\nu:=h_E\ast\rho_\nu$, where $(\rho_\nu)_{\nu\in\mathbb{N}}$ is an approximate identity with respect to $U$.
    By Proposition \ref{vect bdl semi-negative sHm smoothing}, $h_E^\nu$ is smooth Nakano semi-negative Hermitian metric on $E$ over $U^\nu$.
    For any $n$-tuple locally holomorphic sections $u=(u_1,\cdots,u_n)$ of $E$, i.e. $u_j\in H^0(U,E)$, we get
    \begin{align*}
        (u_j,u_k)_{h_E^\nu}(z)=\int (u_j,u_k)_{h_E^{(w)}}(z)\rho_\nu(w)dV_w, \quad
        T^{h_E^\nu}_u(z)=\int T^{h_E^{(w)}}_u(z)\rho_\nu(w)dV_w,
    \end{align*}
    where $h_E^{(w)}(z)=h_E(z-w)$.
    By Nakano semi-negativity of $h\otimes h_E^\nu$, for any test form $\phi$ we have that $\idd T_u^{h\otimes h_E^\nu}(\phi)\geq0$ and that
    \begin{align*}
        0\leq\lim_{\nu\to+\infty}\idd T_u^{h\otimes h_E^\nu}(\phi)&=\lim\int\phi\idd T_u^{h\otimes h_E^\nu}=\lim\int T_u^{h\otimes h_E^\nu}\wedge\idd\phi
        =\lim\int h\cdot T_u^{h_E^\nu}\wedge\idd\phi\\
        &=\lim\int_z \Bigl\{ h\cdot \int_wT^{h_E^{(w)}}_u(z)\rho_\nu(w)dV_w\Bigr\}\wedge\idd\phi\\
        &=\lim\int_w \Bigl\{ \int_z h\cdot T^{h_E^{(w)}}_u(z)\wedge\idd\phi\Bigr\}\rho_\nu(w)dV_w\\
        &=\lim\int_w \Bigl\{ \int_z T^{h\otimes h_E^{(w)}}_u(z)\wedge\idd\phi\Bigr\}\rho_\nu(w)dV_w\\
        &=\lim\int_w \idd T^{h\otimes h_E^{(w)}}_u(\phi)\rho_\nu(w)dV_w\\
        &=\lim \langle \rho_\nu,\idd T^{h\otimes h_E^{(\bullet)}}_u(\phi)\rangle=\langle \delta_0,\idd T^{h\otimes h_E^{(\bullet)}}_u(\phi)\rangle\\
        &=\idd T^{h\otimes h_E^{(0)}}_u(\phi)=\idd T^{h\otimes h_E}_u(\phi),
    \end{align*}
    i.e. $\idd T^{h\otimes h_E}_u\geq0$, where the function $F=\idd T_u^{h\otimes h_E^{(\bullet)}}(\phi):\mathrm{int}(\mathrm{supp}\,\rho_\nu)\to\mathbb{R}$ is continuous near $0$ by smoothness of $h$, similar to the proof of Proposition \ref{key Prop of sing Nakano semi-negative}.

    Finally, we show that $(h_E\otimes h_L)\ast\rho_\nu$ is Nakano semi-negative.
    Let $h_L^\mu:=h_L\ast\rho_\mu$ then $(h_L^\mu)_{\mu\in\mathbb{N}}$ is a sequence of smooth semi-negative Hermitian metrics decreasing to $h_L$ a.e. pointwise by Griffiths semi-negativity of $h_L$ and Proposition \ref{vect bdl semi-negative sHm smoothing}.
    By the above, the sequence of singular Hermitian metrics $(h_E\otimes h_L^\mu)_{\mu\in\mathbb{N}}$ is a sequence of Nakano semi-negative Hermitian metrics decreasing to $h_E\otimes h_L$ a.e. pointwise.

    Therefore, for any locally constant section $s\in\mathcal{O}_x(E\otimes L)$ and any positive integers $\lambda>\mu$, we get the inequality $||s||^2_{h_E\otimes h_L^\mu}\geq||s||^2_{h_E\otimes h_L^\lambda}$, i.e. $f_{s,\mu,\lambda}:=||s||^2_{h_E\otimes h^\mu}-||s||^2_{h_E\otimes h^\lambda}\geq0.$
    In particular, $f_{s,\mu,+\infty}=||s||^2_{h_E\otimes h_L^\mu}-||s||^2_{h_E\otimes h_L}\geq0$ a.e. as the case where $\lambda=+\infty$.

    We fixed a positive integer $\nu$. For any positive integers $\lambda>\mu$, we have that
    \begin{align*}
        0\leq f_{s,\mu,\lambda}\ast\rho_\nu&=(||s||^2_{h_E\otimes h_L^\mu}-||s||^2_{h_E\otimes h_L^\lambda})\ast\rho_\nu
        =||s||^2_{(h_E\otimes h_L^\mu)\ast\rho_\nu}-||s||^2_{(h_E\otimes h_L^\lambda)\ast\rho_\nu}
    \end{align*}
    and that $0\leq f_{s,\mu,+\infty}\ast\rho_\nu=||s||^2_{(h_E\otimes h_L^\mu)\ast\rho_\nu}-||s||^2_{(h_E\otimes h_L)\ast\rho_\nu}$.
    From reverse Fatou's lemma, the decreasing sequence of smooth semi-positive functions $(f_{s,\mu,+\infty}\ast\rho_\nu)_{\mu\in\mathbb{N}}$ converges to $0$ pointwise.
    In fact, $f_{s,\mu,+\infty}\ast\rho_\nu-f_{s,\mu+1,+\infty}\ast\rho_\nu=f_{s,\mu,\mu+1}\ast\rho_\nu\geq0$ and 
    \begin{align*}
        0&\leq\lim_{\mu\to+\infty}f_{s,\mu,+\infty}\ast\rho_\nu(z)=\lim_{\mu\to+\infty}\int f_{s,\mu,+\infty}(z-w)\rho_\nu(w)dV_w\\
        &\leq\limsup_{\mu\to+\infty}\int f_{s,\mu,+\infty}(z-w)\rho_\nu(w)dV_w\leq\int\limsup_{\mu\to+\infty} f_{s,\mu,+\infty}(z-w)\rho_\nu(w)dV_w=0,
    \end{align*}
    where $0\leq f_{s,1,+\infty}\leq||s||^2_{h_E\otimes h_L^1}=h_L^1||s||^2_{h_E}$ is locally integrable by smoothness of $h_L^1$ and plurisubharmonicity of $||s||^2_{h_E}$.

    Hence, the sequence of smooth Hermitian metrics $((h_E\otimes h_L^\mu)\ast\rho_\nu)_{\mu\in\mathbb{N}}$ decreases to $(h_E\otimes h_L)\ast\rho_\nu$ pointwise and each metric $(h_E\otimes h_L^\mu)\ast\rho_\nu$ is smooth Nakano semi-negative by $(b)$ of Proposition \ref{vect bdl semi-negative sHm smoothing}. 
    Thus $(h_E\otimes h_L)\ast\rho_\nu$ is smooth Nakano semi-negative for each $\nu$ (see [38,\,Corollary\,5.6\, and \,Theorem\,1.7]).    
\end{proof}


\begin{corollary}\label{s d-posi * Nak semi-posi then s Nak d-posi}
    Let $X$ be a \kah manifold and $\omega$ be a \kah metric. 
    Let $L$ and $E$ be a holomorphic line bundle and a holomorphic vector bundle over $X$ equipped with singular Hermitian metrics $h_L$ and $h_E$, respectively.
    We have the following
    \begin{itemize}
        \item [$(a)$] If $h_L$ is singular semi-positive and $h_E$ is $L^2$-type strictly Nakano $\delta_\omega$-positive then $h_E\otimes h_L$ is also $L^2$-type strictly Nakano $\delta_\omega$-positive.
        \item [$(b)$] If $h_L$ is strictly $\delta_\omega$-positive 
        and $h_E$ is $L^2$-type Nakano semi-positive then $h_E\otimes h_L$ is $L^2$-type strictly Nakano $\delta_\omega$-positive.
        \item [$(c)$] If $h_L$ is singular semi-positive and $h_E$ is strictly dual Nakano $\delta_\omega$-positive then $h_E\otimes h_L$ is also strictly dual Nakano $\delta_\omega$-positive.
        \item [$(d)$] If $h_L$ is strictly $\delta_\omega$-positive 
        and $h_E$ is dual Nakano semi-positive then $h_E\otimes h_L$ is strictly dual Nakano $\delta_\omega$-positive.
    \end{itemize}
\end{corollary}

Since Griffiths and dual Nakano semi-positivity are local properties, we get following.

\begin{proposition}\label{s-positivity for two metric}
    Let $X$ be a \kah manifold and $\omega,\gamma$ be \kah metrics on $X$.
    Let $E$ be a holomorphic vector bundle over $X$ equipped with a singular Hermitian metric $h$. 
    We assume that there exists a positive number $c>0$ such that $\omega\geq c\gamma$. Then we get 
    \begin{itemize}
        \item [$(a)$] If $h$ is strictly Griffiths $\delta_\omega$-positive then 
        $h$ is strictly Griffiths $c\delta_\gamma$-positive.
        \item [$(b)$] If $h$ is strictly dual Nakano $\delta_\omega$-positive then 
        $h$ is strictly dual Nakano $c\delta_\gamma$-positive.
    \end{itemize}
\end{proposition}

\begin{proof}
    $(b)$ We show that 
    for any open subset $U$ and any \kah potential $\psi$ of $\gamma$ on $U$, $he^{c\delta\psi}$ is dual Nakano semi-positive on $U$.
    Since dual Nakano semi-positivity is a local property, it is sufficient to show that for any $x_0\in U$, there exists a neighborhood $B$ of $x_0$ such that $B\subset U$ and $he^{c\delta\psi}$ is dual Nakano semi-positive on $B$.
    Here, we take $B$ such that the \kah potential $\varphi$ of $\omega$ on $B$ exists. Then $\varphi-c\psi$ is plurisubharmonic by $\omega\geq c\gamma$ and $e^{-\delta(\varphi-c\psi)}$ is semi-positive Hermitian metric on trivial line bundle.
    From the assumption, we have that $he^{\delta\varphi}$ is dual Nakano semi-positive on $B$.
    By $(c)$ of Theorem \ref{psef * Nak semi-posi then Nak semi-posi}, we have that $he^{\delta\varphi}\otimes e^{-\delta(\varphi-c\psi)}=he^{c\delta\psi}$ is dual Nakano semi-positive on $B$.

    $(a)$ It is shown in the same way as $(b)$.
\end{proof}

Finally, we obtain the following dual-type generalization of Demailly and Skoda's theorem \cite{DS80} to singularities.

\begin{theorem}$\mathrm{(=Theorem\,\ref{Grif * det Grif is dual Nakano})}$
    Let $X$ be a complex manifold and $E$ be a holomorphic vector bundle over $X$ equipped with a singular Hermitian metric $h$.
    If $h$ is Griffiths semi-positive then $h\otimes\mathrm{det}\,h$ is dual Nakano semi-positive.
\end{theorem}

\begin{proof}
    For simplicity, we prove that if $h$ is Griffiths semi-negative then $h\otimes\mathrm{det}\,h$ is Nakano semi-negative.
    Since Proposition \ref{key Prop of sing Nakano semi-negative}, it is sufficient to show that $(h\otimes\mathrm{det}\,h)\ast\rho_\nu$ is smooth Nakano semi-negative for each $\nu\in\mathbb{N}$.

    By Griffiths semi-negativity of $h^{\mu}:=h\ast\rho_\mu$, we get smooth Nakano semi-negativity of $h^\mu\otimes\mathrm{det}\,h^\mu$.
    Moreover, the sequence of smooth Nakano semi-negative Hermitian metrics $(h^\mu\otimes\mathrm{det}\,h^\mu)_{\mu\in\mathbb{N}}$ decreases and converges to $h\otimes\mathrm{det}\,h$ a.e. pointwise by $(h^\mu)_{\mu\in\mathbb{N}}$ decreasing to $h$. 
    Therefore, for any locally constant section $s\in\mathcal{O}_x(E\otimes\mathrm{det}\,E)$ and any two positive integers $\lambda>\mu$, we get the inequality 
    \begin{align*}
        ||s||^2_{h^\mu\otimes\mathrm{det}\,h^\mu}\geq||s||^2_{h^\lambda\otimes\mathrm{det}\,h^\lambda} \quad \mathrm{i.e.} \quad f_{s,\mu,\lambda}:=||s||^2_{h^\mu\otimes\mathrm{det}\,h^\mu}-||s||^2_{h^\lambda\otimes\mathrm{det}\,h^\lambda}\geq0.
    \end{align*}
    In particular, $f_{s,\mu,+\infty}=||s||^2_{h^\mu\otimes\mathrm{det}\,h^\mu}-||s||^2_{h\otimes\mathrm{det}\,h}\geq0$ a.e. as the case where $\lambda=+\infty$.

    From reverse Fatou's lemma, the decreasing sequence of smooth semi-positive functions $(f_{s,\mu,+\infty}\ast\rho_\nu)_{\mu\in\mathbb{N}}$ converges to $0$ pointwise, a similar to the proof of $(c)$ of Theorem \ref{psef * Nak semi-posi then Nak semi-posi}.
    Hence, the sequence of smooth Hermitian metrics $((h^\mu\otimes\mathrm{det}\,h^\mu)\ast\rho_\nu)_{\mu\in\mathbb{N}}$ decreases to $(h\otimes\mathrm{det}\,h)\ast\rho_\nu$ pointwise and each metric $(h^\mu\otimes\mathrm{det}\,h^\mu)\ast\rho_\nu$ is smooth Nakano semi-negative by $(b)$ of Proposition \ref{vect bdl semi-negative sHm smoothing}.  
    Thus $(h\otimes\mathrm{det}\,h)\ast\rho_\nu$ is smooth Nakano semi-negative for each $\nu$ (see [38,\,Corollary\,5.6\, and \,Theorem\,1.7]). 
\end{proof}

\section{$L^2$-estimates with singular Hermitian metrics}

\subsection{$L^2$-estimates for line bundles possessing singular Hermitian metrics}

In this subsection, we show $L^2$-estimates on weakly pseudoconvex \kah manifolds when a holomorphic line bundle has a singular (semi)-positive Hermitian metric.
First, we give semi-positive case using following lemmas and Demailly's approximation.

\begin{lemma}\label{complete kah on X_j-Z}$\mathrm{(cf.\,[4,\,Theorem\,1.5]})$ 
    Let $X$ be a \kah manifold and $Z$ be an analytic subset of $X$. Assume that $\Omega$ is a relatively open subset of $X$ possessing a complete \kah metric. Then $\Omega\setminus Z$ carries a complete \kah metric.
\end{lemma}

\begin{lemma}\label{Ext d-equation for hypersurface}$\mathrm{(cf.\,[4,\,Lemma\,6.9])}$  
    Let $\Omega$ be an open subset of $\mathbb{C}^n$ and $Z$ be a complex analytic subset of $\Omega$. Assume that $u$ is a $(p,q-1)$-form with $L^2_{loc}$ coefficients and $g$ is a $(p,q)$-form with $L^1_{loc}$ coefficients such that $\overline{\partial}u=g$ on $\Omega\setminus Z$. 
    Then $\overline{\partial}u=g$ on $\Omega$.
\end{lemma}

\begin{theorem}\label{L2-estimate of m-posi + psef}
    Let $X$ be a weakly pseudoconvex \kah manifold and $\omega$ be a \kah metric on $X$.
    Let $(F,h_F)$ be a holomorphic vector bundle of rank $r$ and $L$ be a holomorphic line bundle equipped with a singular semi-positive Hermitian metric $h$, i.e. $i\Theta_{L,h}\geq0$ in the sense of currents.
    Then we have the following
    \begin{itemize}
        \item [($a$)] If $h_F$ is $m$-positive, then for any $q\geq1$ with $m\geq\min\{n-q+1,r\}$ and any $f\in L^2_{n,q}(X,F\otimes L,h_F\otimes h,\omega)$ satisfying $\overline{\partial}f=0$ and 
        $\int_X\langle B^{-1}_{h_F,\omega}f,f\rangle_{h_F\otimes h,\omega}dV_\omega<+\infty$,
        there exists $u\in L^2_{n,q-1}(X,F\otimes L,h_F\otimes h,\omega)$ such that $\overline{\partial}u=f$ and 
        \begin{align*}
            \int_X|u|^2_{h_F\otimes h,\omega}dV_\omega\leq\int_X\langle B^{-1}_{h_F,\omega}f,f\rangle_{h_F\otimes h,\omega}dV_\omega,
        \end{align*}
    \end{itemize}
    \begin{itemize}
        \item [($b$)] If $h_F$ is dual $m$-positive, then for any $p\geq1$ with $m\geq\min\{n-p+1,r\}$ and any $f\in L^2_{p,n}(X,F\otimes L,h_F\otimes h,\omega)$ satisfying $\overline{\partial}f=0$ and 
        $\int_X\langle B^{-1}_{h_F,\omega}f,f\rangle_{h_F\otimes h,\omega}dV_\omega<+\infty$,
        there exists $u\in L^2_{p,n-1}(X,F\otimes L,h_F\otimes h,\omega)$ such that $\overline{\partial}u=f$ and 
        \[
            \int_X|u|^2_{h_F\otimes h,\omega}dV_\omega\leq\int_X\langle B^{-1}_{h_F,\omega}f,f\rangle_{h_F\otimes h,\omega}dV_\omega,
        \]
    \end{itemize}
    where $B_{h_F,\omega}=[i\Theta_{F,h_F}\otimes\mathrm{id}_L,\Lambda_\omega]$.
\end{theorem}

\begin{proof}
    There exists a smooth exhaustive plurisubharmonic function $\psi$ on $X$ such that $\sup_X\psi=+\infty$. Let $X_j:=\{x\in X\mid \psi(z)<j\}$ be a sublevel set which is weakly pseudoconvex. 
    Let $h_0$ be a smooth Hermitian metric on $L$ then $h=h_0e^{-\varphi}$, where $\varphi$ is quasi-plurisubharmonic function on $X$ and $i\Theta_{L,h}=i\Theta_{L,h_0}+\idd\varphi\geq0$ as currents.

    By Theorem \ref{Demailly smoothing of current}, there is a sequence of quasi-plurisubharmonic functions $(\varphi_\nu)_{\nu\in\mathbb{N}}$ defined on $X_j$ such that 
    \begin{itemize}
        \item [$(i)$] $\varphi_\nu$ is smooth in the complement $X_j\setminus Z_\nu$ of an analytic set $Z_\nu\subset X_j$,
        \item [$(ii)$] $(\varphi_\nu)_{\nu\in\mathbb{N}}$ is a decreasing sequence and $\varphi|_{X_j}=\lim_{\nu\to+\infty}\varphi_\nu$,
        \item [$(iii)$] $i\Theta_{L,h_0}+\idd\varphi_\nu\geq-\beta_\nu\omega$, where $\lim_{\nu\to+\infty}\beta_\nu=0$.
    \end{itemize}
    Here, we can find a sequence of Hermitian metric $h_\nu=h_0e^{-\varphi_\nu}$ on $L|_{X_j}$. Then $h_\nu$ is smooth on $X_j\setminus Z_\nu$, $h_\nu\leq h$ and $i\Theta_{L,h_\nu}\geq-\beta_\nu\omega$.

    $(a)$ By $m$-positivity of $h_F$, Lemma \ref{m-posi then A>0 and dual m-posi then A>0} and Proposition \ref{A_(E+F)>C_E+C_F if A_E>C_E and A_F>C_F}, we get $A^{n,q}_{F,h_F,\omega}>0$ on $X$ and $B_{h_F,\omega}=[i\Theta_{F,h_F}\otimes\mathrm{id}_L,\Lambda_\omega]>0$ for $q\geq1$ with $m\geq\min\{n-q+1,r\}$.
    Fixed a positive integer $q\geq1$ with $m\geq\min\{n-q+1,r\}$. 
    From $B_{h_F,\omega}>0$ on $X$, $\lim_{\nu\to+\infty}\beta_\nu=0$ and relatively compact-ness of $X_j$, there exist $c>0$ and $\nu_0\in\mathbb{N}$ such that for any $\nu\geq\nu_0$, $0<q\beta_\nu<c<B_{h_F,\omega}$ on $X_j$. 
    Then by smooth-ness of $h_\nu$, we have that 
    \begin{align*}
        A^{n,q}_{F\otimes L,h_F\otimes h_\nu,\omega}&=[i\Theta_{F,h_F}\otimes\mathrm{id}_L,\Lambda_\omega]+[i\Theta_{L,h_\nu}\otimes\mathrm{id}_F,\Lambda_\omega]\\
        &\geq B_{h_F,\omega}-\beta_\nu[\omega\otimes\mathrm{id}_F,\Lambda_\omega]=B_{h_F,\omega}-q\beta_\nu\\
        &\geq\bigl(1-\frac{q\beta_\nu}{c}\bigr)B_{h_F,\omega}>0
    \end{align*}
    on $X_j\setminus Z_\nu$. Hence, for any $f\in L^2_{n,q}(X,F\otimes L,h_F\otimes h,\omega)$ satisfying $\overline{\partial}f=0$ and $\int_X\langle B^{-1}_{h_F,\omega}f,f\rangle_{h_F\otimes h,\omega}dV_\omega<+\infty$, we have that
    \begin{align*}
        \int_{X_j\setminus Z_\nu}\langle[i\Theta_{F\otimes L,h_F\otimes h_\nu},\Lambda_\omega]^{-1}f,f\rangle_{h_F\otimes h_\nu,\omega}dV_\omega&\leq\frac{c}{c-q\beta_\nu}\int_{X_j\setminus Z_\nu}\langle B_{h_F,\omega}^{-1}f,f\rangle_{h_F\otimes h_\nu,\omega}dV_\omega\\
        &\leq\frac{c}{c-q\beta_\nu}\int_X\langle B_{h_F,\omega}^{-1}f,f\rangle_{h_F\otimes h,\omega}dV_\omega<+\infty.
    \end{align*}

    By Lemma \ref{complete kah on X_j-Z}, $X_j\setminus Z_\nu$ has a complete \kah metric. From [8,\,ChapterVIII,\,Theorem\,6.1], 
    i.e. $L^2$-estimates for $(n,q)$-forms without completeness of $\omega$, we obtain a solution $u_{j,\nu}\in L^2_{n,q-1}(X_j\setminus Z_\nu,F\otimes L,h_F\otimes h_\nu,\omega)$ of $\overline{\partial}u_{j,\nu}=f$ on $X_j\setminus Z_\nu$ satisfying 
    \begin{align*}
        \int_{X_j\setminus Z_\nu}|u_{j,\nu}|^2_{h_F\otimes h_\nu}dV_\omega  &\leq\int_{X_j\setminus Z_\nu}\langle[i\Theta_{F\otimes L,h_F\otimes h_\nu},\Lambda_\omega]^{-1}f,f\rangle_{h_F\otimes h_\nu,\omega}dV_\omega\\
        &\leq\frac{c}{c-q\beta_\nu}\int_X\langle B_{h_F,\omega}^{-1}f,f\rangle_{h_F\otimes h,\omega}dV_\omega<+\infty.
    \end{align*}

    From Lemma \ref{Ext d-equation for hypersurface}, letting $u_{j,\nu}=0$ on $Z_\nu$ then we obtain that $u_{j,\nu}\in L^2_{n,q-1}(X_j,F\otimes L,h_F\otimes h_\nu,\omega)$, $\overline{\partial}u_{j,\nu}=f$ on $X_j$ and that 
    \begin{align*}
        \bigl(1-\frac{q\beta_\nu}{c}\bigr)\int_{X_j}|u_{j,\nu}|^2_{h_F\otimes h_\nu}dV_\omega\leq\int_X\langle B_{h_F,\omega}^{-1}f,f\rangle_{h_F\otimes h,\omega}dV_\omega.
    \end{align*}

    Since the monotonicity to $\nu$ of $|\bullet|^2_{h_F\otimes h_\nu,\omega}$ by $(ii)$, we get a solution $u\in L^2_{n,q-1}(X,F\otimes L,h_F\otimes h,\omega)$ of $\overline{\partial}u=f$ on $X$ such that 
    \begin{align*}
        \int_X|u|^2_{h_F\otimes h}dV_\omega\leq\int_X\langle B_{h_F,\omega}^{-1}f,f\rangle_{h_F\otimes h,\omega}dV_\omega<+\infty,
    \end{align*}
    using the same method as in the proof of $(b)$ of Theorem \ref{psef * Nak semi-posi then Nak semi-posi}.

    $(b)$ 
    It is shown in the same way as above using [39,\,Theorem\,3.7]. 
\end{proof}


\begin{corollary}\label{L2-estimate of k-posi + psef}
    Let $X$ be a weakly pseudoconvex \kah manifold and $\omega$ be a \kah metric on $X$.
    Let $(A,h_A)$ be a $k$-positive holomorphic line bundle and $L$ be a holomorphic line bundle equipped with a singular semi-positive Hermitian metric $h$, i.e. $i\Theta_{L,h}\geq0$ in the sense of currents.
    Then we have the following
    \begin{itemize}
        \item [($a$)] For any $q\geq k$ and any $f\in L^2_{n,q}(X,A\otimes L,h_A\otimes h,\omega)$ satisfying $\overline{\partial}f=0$ and 
        $\int_X\langle B^{-1}_{h_A,\omega}f,f\rangle_{h_A\otimes h,\omega}dV_\omega<+\infty$,
        there exists $u\in L^2_{n,q-1}(X,A\otimes L,h_A\otimes h,\omega)$ such that $\overline{\partial}u=f$ and 
        \begin{align*}
            \int_X|u|^2_{h_A\otimes h,\omega}dV_\omega\leq\int_X\langle B^{-1}_{h_A,\omega}f,f\rangle_{h_A\otimes h,\omega}dV_\omega,
        \end{align*}
        \item [($b$)] For any $p\geq k$ and any $f\in L^2_{p,n}(X,A\otimes L,h_A\otimes h,\omega)$ satisfying $\overline{\partial}f=0$ and 
        $\int_X\langle B^{-1}_{h_A,\omega}f,f\rangle_{h_A\otimes h,\omega}dV_\omega<+\infty$,
        there exists $u\in L^2_{p,n-1}(X,A\otimes L,h_A\otimes h,\omega)$ such that $\overline{\partial}u=f$ and 
        \begin{align*}
            \int_X|u|^2_{h_A\otimes h,\omega}dV_\omega\leq\int_X\langle B^{-1}_{h_A,\omega}f,f\rangle_{h_A\otimes h,\omega}dV_\omega,
        \end{align*}
    \end{itemize}
    where $B_{h_A,\omega}=[i\Theta_{A,h_A}\otimes\mathrm{id}_L,\Lambda_\omega]$.
\end{corollary}

Here, $A^{p,q}_{L,h,\omega}>0$ for $p+q\geq n+k$ if $h$ is $k$-positive.
Second, we obtain $L^2$-estimates when singular Hermitian metrics 
have positivity by using the following proposition.

\begin{proposition}\label{big of bounded below}
    Let $X$ be a weakly pseudoconvex \kah manifold and $\omega$ be a \kah metric on $X$.
    Let $L$ be a holomorphic line bundle over $X$ equipped with a singular positive Hermitian metric $h$.
    Then there exists a positive smooth function $c:X\to\mathbb{R}_{>0}$ such that $i\Theta_{L,h}\geq c\omega$ in the sense of currents.
\end{proposition}

This proposition can be shown by taking a relatively compact increasing sequence of $X$.
Similar to the proof of Theorem \ref{L2-estimate of m-posi + psef}, we get the following theorem. 

\begin{theorem}\label{L2-estimate of m-semi-posi + big}
    Let $X$ be a weakly pseudoconvex \kah manifold and $\omega$ be a \kah metric on $X$.
    Let $(F,h_F)$ be a holomorphic vector bundle of rank $r$ and $L$ be a holomorphic line bundle equipped with a singular positive Hermitian metric $h$.
    Here, there exists a positive smooth function $c:X\to\mathbb{R}_{>0}$ such that $i\Theta_{L,h}\geq 2c\omega$ in the sense of currents by Proposition \ref{big of bounded below}.
    Then we have the following
    \begin{itemize}
        \item [($a$)] If $h_F$ is $m$-semi-positive, then for any $q\geq1$ with $m\geq\min\{n-q+1,r\}$ and any $f\in L^2_{n,q}(X,F\otimes L,h_F\otimes h,\omega)$ satisfying $\overline{\partial}f=0$
        and $\int_X\frac{1}{c}|f|^2_{h_F\otimes h,\omega}dV_\omega<+\infty$,
        there exists $u\in L^2_{n,q-1}(X,F\otimes L,h_F\otimes h,\omega)$ such that $\overline{\partial}u=f$ and 
        \begin{align*}
            \int_X|u|^2_{h_F\otimes h,\omega}dV_\omega\leq\frac{1}{q}\int_X\frac{1}{c}|f|^2_{h_F\otimes h,\omega}dV_\omega.
        \end{align*}
        \item [($b$)] If $h_F$ is dual $m$-semi-positive, then for any $p\geq1$ with $m\geq\min\{n-p+1,r\}$ and any $f\in L^2_{p,n}(X,F\otimes L,h_F\otimes h,\omega)$ satisfying $\overline{\partial}f=0$
        and $\int_X\frac{1}{c}|f|^2_{h_F\otimes h,\omega}dV_\omega<+\infty$,
        there exists $u\in L^2_{p,n-1}(X,F\otimes L,h_F\otimes h,\omega)$ such that $\overline{\partial}u=f$ and 
        \begin{align*}
            \int_X|u|^2_{h_F\otimes h,\omega}dV_\omega\leq\frac{1}{p}\int_X\frac{1}{c}|f|^2_{h_F\otimes h,\omega}dV_\omega.
        \end{align*}
    \end{itemize}
\end{theorem}

\subsection{$L^2$-estimates with singular (dual) Nakano semi-positivity}

In this subsection, 
we obtain $L^2$-estimates on weakly pseudoconvex \kah manifolds with a positive line bundle for two cases where the singular Hermitian metric of holomorphic vector bundles has $L^2$-type Nakano semi-positivity and dual Nakano semi-positivity.

\begin{theorem}\label{L2-estimates of L2-type Nakano semi-posi + m-posi on w.p.c}
    Let $X$ be a weakly pseudoconvex manifold and 
    $E$ be a holomorphic vector bundle equipped with a singular Hermitian metric $h$ and $\omega$ be a \kah metric. 
    We assume that $h$ is $L^2$-type Nakano semi-positive on $X$. Then we have the following
    \begin{itemize}
        \item [($a$)] If $X$ has a positive holomorphic line bundle and $(A,h_A)$ is a $k$-positive line bundle, then for any $q\geq k$ and any $f\in L^2_{n,q}(X,A\otimes E,h_A\otimes h,\omega)$ satisfying $\overline{\partial}f=0$ and $\int_X\langle B^{-1}_{h_A,\omega}f,f\rangle_{h_A\otimes h,\omega}dV_\omega<+\infty$,
        there exists $u\in L^2_{n,q-1}(X,A\otimes E,h_A\otimes h,\omega)$ satisfies $\overline{\partial}u=f$ and 
        \begin{align*}
            \int_X|u|^2_{h_A\otimes h,\omega}dV_{\omega}\leq\int_X\langle B^{-1}_{h_A,\omega}f,f\rangle_{h_A\otimes h,\omega}dV_\omega,
        \end{align*}
        where $B_{h_A,\omega}=[i\Theta_{A,h_A}\otimes\mathrm{id}_E,\Lambda_\omega]$.
        \item [($b$)] If $(F,h_F)$ is a $m$-positive holomorphic vector bundle of rank $r$, then for any $q\geq1$ with $m\geq\min\{n-q+1,r\}$ and any $f\in L^2_{n,q}(X,F\otimes E,h_F\otimes h,\omega)$ satisfying $\overline{\partial}f=0$ and $\int_X\langle B^{-1}_{h_F,\omega}f,f\rangle_{h_F\otimes h,\omega}dV_\omega<+\infty$,
        there exists $u\in L^2_{n,q-1}(X,F\otimes E,h_F\otimes h,\omega)$ satisfies $\overline{\partial}u=f$ and 
        \begin{align*}
            \int_X|u|^2_{h_F\otimes h,\omega}dV_{\omega}\leq\int_X\langle B^{-1}_{h_F,\omega}f,f\rangle_{h_F\otimes h,\omega}dV_\omega.
        \end{align*}
        where $B_{h_F,\omega}=[i\Theta_{F,h_F}\otimes\mathrm{id}_E,\Lambda_\omega]$.
    \end{itemize}
\end{theorem}

\begin{proof}
    $(b)$ There exists a smooth exhaustive plurisubharmonic function $\Psi$ on $X$ such that $\sup_X\Psi=+\infty$. Let $X_j:=\{x\in X\mid\Psi(z)<j\}$ be a sublevel set.
    From $m$-positivity of $h_F$ and Lemma \ref{m-posi then A>0 and dual m-posi then A>0}, we get $A^{n,q}_{F,h_F,\omega}>0$ on $X$ for $q\geq1$ with $m\geq\min\{n-q+1,r\}$ and that $(\mathrm{det}\,F,\mathrm{det}\,h_F)$ is a positive line bundle.
    By [36,\,Theorem 1.2], there exists a holomorphically embedding map $\Phi:X\to\mathbb{P}^{2n+1}$. Here, there exists a hypersurface $H$ on $\mathbb{P}^{2n+1}$ such that $\Phi(X)\cap H\ne\emptyset$ and that $\mathbb{P}^{2n+1}\setminus H$ is a Stein subset.
    In particular, there is a strictly plurisubharmonic function $\psi$ on $\mathbb{P}^{2n+1}\setminus H$ which is smooth and exhaustive, i.e. $\psi(z)\nearrow+\infty$ as $z$ tends to $H$.
    Then we have that $X_j\setminus \Phi^{-1}(H)$ is Stein submanifold. 
    
    In fact, the smooth function $\Phi^*\psi$ on $X\setminus\Phi^{-1}(H)$ is also strictly plurisubharmonic and satisfies $\Phi^*\psi(z)\nearrow+\infty$ as $z$ tends to $\Phi^{-1}(H)$. Therefore, the smooth function $\Phi^*\psi-\log(j-\Psi)$ on $X_j\setminus \Phi^{-1}(H)$ is strictly plurisubharmonic and exhaustive.

    Fixed a positive integer $q\geq1$ with $m\geq\min\{n-q+1,r\}$. Since $L^2$-type Nakano semi-positivity of $h$,
    for any $f\in L^2_{n,q}(X,F\otimes E,h_F\otimes h,\omega)$ satisfying $\overline{\partial}f=0$ and $\int_X\langle B^{-1}_{h_F,\omega}f,f\rangle_{h_F\otimes h,\omega} dV_\omega<+\infty$, 
    there exists $u_j\in L^2_{n,q-1}(X_j\setminus\Phi^{-1}(H),F\otimes E,h_F\otimes h,\omega)$ satisfies $\overline{\partial}u_j=f$ on $X_j\setminus\Phi^{-1}(H)$ and 
    \begin{align*}
        \int_{X_j\setminus\Phi^{-1}(H)}|u_j|^2_{h_F\otimes h,\omega}dV_{\omega}&\leq\int_{X_j\setminus\Phi^{-1}(H)}\langle B^{-1}_{h_F,\omega}f,f\rangle_{h_F\otimes h,\omega}dV_\omega\leq\int_X\langle B^{-1}_{h_F,\omega}f,f\rangle_{h_F\otimes h,\omega}dV_\omega,
    \end{align*}
    where Lebesgue measure of $\Phi^{-1}(H)$ is zero.

    Let $u_j=0$ on $\Phi^{-1}(H)$ then by Lemma \ref{Ext d-equation for hypersurface},
    we have that $u_j\in L^2_{n,q-1}(X_j,F\otimes E,h_F\otimes h,\omega)$, $\overline{\partial}u_j=f$ on $X_j$ and 
    \begin{align*}
        \int_{X_j}|u_j|^2_{h_F\otimes h,\omega}dV_\omega\leq\int_X\langle B^{-1}_{h_F,\omega}f,f\rangle_{h_F\otimes h,\omega}dV_\omega.
    \end{align*}

    Hence, we get a solution $u\in L^2_{n,q-1}(X,F\otimes E,h_F\otimes h,\omega)$ of $\overline{\partial}u=f$ on $X$ such that 
    \begin{align*}
        \int_X|u|^2_{h_F\otimes h,\omega}dV_\omega\leq\int_X\langle B^{-1}_{h_F,\omega}f,f\rangle_{h_F\otimes h,\omega}dV_\omega,
    \end{align*}
    using the same method as in the proof of $(b)$ of Theorem \ref{psef * Nak semi-posi then Nak semi-posi}.

    $(a)$ It is shown from the fact $A^{p,q}_{A,h_A,\omega}>0$ for $p+q>n+k-1$ as above.
\end{proof}

\begin{theorem}\label{L2-estimates of dual Nakano semi-posi + m-posi on w.p.c}
    Let $X$ be a weakly pseudoconvex manifold and 
    $E$ be a holomorphic vector bundle equipped with a singular Hermitian metric $h$ and $\omega$ be a \kah metric.
    We assume that $h$ is dual Nakano semi-positive on $X$. Then we have the following
    \begin{itemize}
        \item [($a$)] If $X$ has a positive holomorphic line bundle and $(A,h_A)$ is a $k$-positive line bundle, then for any $p\geq k$ and any $f\in L^2_{p,n}(X,A\otimes E,h_A\otimes h,\omega)$ satisfying $\overline{\partial}f=0$ and $\int_X\langle B^{-1}_{h_A,\omega}f,f\rangle_{h_A\otimes h,\omega}dV_\omega<+\infty$,
        there exists $u\in L^2_{p,n-1}(X,A\otimes E,h_A\otimes h,\omega)$ satisfies $\overline{\partial}u=f$ and 
        \begin{align*}
            \int_X|u|^2_{h_A\otimes h,\omega}dV_{\omega}\leq\int_X\langle B^{-1}_{h_A,\omega}f,f\rangle_{h_A\otimes h,\omega}dV_\omega,
        \end{align*}
        where $B_{h_A,\omega}=[i\Theta_{A,h_A}\otimes\mathrm{id}_E,\Lambda_\omega]$.
        \item [($b$)] If $(F,h_F)$ is a dual $m$-positive holomorphic vector bundle of rank $r$, then for any $p\geq1$ with $m\geq\min\{n-p+1,r\}$ and any $f\in L^2_{p,n}(X,F\otimes E,h_F\otimes h,\omega)$ satisfying $\overline{\partial}f=0$ and $\int_X\langle B^{-1}_{h_F,\omega}f,f\rangle_{h_F\otimes h,\omega}dV_\omega<+\infty$,
        there exists $u\in L^2_{p,n-1}(X,F\otimes E,h_F\otimes h,\omega)$ satisfies $\overline{\partial}u=f$ and 
        \begin{align*}
            \int_X|u|^2_{h_F\otimes h,\omega}dV_{\omega}\leq\int_X\langle B^{-1}_{h_F,\omega}f,f\rangle_{h_F\otimes h,\omega}dV_\omega.
        \end{align*}
        where $B_{h_F,\omega}=[i\Theta_{F,h_F}\otimes\mathrm{id}_E,\Lambda_\omega]$.
    \end{itemize}
\end{theorem}

\begin{proof}
    $(b)$ There exists a smooth exhaustive plurisubharmonic function $\Psi$ on $X$ such that $\sup_X\Psi=+\infty$. Let $X_j:=\{x\in X\mid\Psi(z)<j\}$ be a sublevel set.
    From dual $m$-positivity of $h_F$, $(\mathrm{det}\,F,\mathrm{det}\,h_F)$ is a positive line bundle.

    Similarly to the proof of Theorem \ref{L2-estimates of L2-type Nakano semi-posi + m-posi on w.p.c}, there exists an analytic subset $Z$ such that $X_j\setminus Z$ is Stein submanifold for any $j>0$.
    By Lemma \ref{trivial of E on Stein-H}, there exists a hypersurface $H_j$ such that $S_j:=(X_j\setminus Z)\setminus H_j$ is also Stein and $E|_{S_j}$ is trivial. 
    From Steinness of $S_j$, there exists a increasing sequence of open Stein subsets $(S_j(k))_{k\in\mathbb{N}}$ such that $S_j(k)$ is relatively compact. 
    Fiexd $k\in\mathbb{N}$. There exists $\nu_0\in\mathbb{N}$ such that for any integer $\nu\geq\nu_0$, $S_j(k)\subset\subset S^{\nu_0}_j \subset\subset S^\nu_j$ where $S^\nu_j$ is the notation at $\S3.2$. 
    For an approximate identity $(\rho_\nu)_{\nu\in\mathbb{N}}$ with respect to $S_j$, we define the smooth Hermitian metric $h_\nu:=(h^*\ast\rho_\nu)^*$ on $E$ over $S_j^\nu$. Here, $h_\nu$ is dual Nakano semi-positive by Proposition \ref{vect bdl semi-negative sHm smoothing}.

    From dual $m$-positivity of $h_F$, Lemma \ref{m-posi then A>0 and dual m-posi then A>0} and Proposition \ref{A_(E+F)>C_E+C_F if A_E>C_E and A_F>C_F}, we obtain that $A^{p,n}_{F,h_F,\omega}>0$ and $B_{h_F,\omega}=[i\Theta_{F,h_F}\otimes\mathrm{id}_E,\Lambda_\omega]>0$ for $p\geq1$ with $m\geq\min\{n-p+1,r\}$.
    Fixed a positive integer $p\geq1$ with $m\geq\min\{n-p+1,r\}$. By dual Nakano semi-positivity of $h_\nu$ and Proposition \ref{A_(E+F)>C_E+C_F if A_E>C_E and A_F>C_F}, we have that $A^{k,n}_{E,h_\nu,\omega}\geq0$ for $k\geq1$ and that 
    \begin{align*}
        A^{p,n}_{E\otimes F,h_\nu\otimes h_F,\omega}&=[i\Theta_{E,h_\nu}\otimes\mathrm{id}_F,\Lambda_\omega]+[i\Theta_{F,h_F}\otimes\mathrm{id}_E,\Lambda_\omega]\\
        &\geq[i\Theta_{F,h_F}\otimes\mathrm{id}_E,\Lambda_\omega]=B_{h_F,\omega}>0,
    \end{align*}
    i.e. $0<(A^{p,n}_{E\otimes F,h_\nu\otimes h_F,\omega})^{-1}\leq B_{h_F,\omega}^{-1}$ on $S_j(k)$ for any $\nu\geq\nu_0$.
    For any $\overline{\partial}$-closed $f\in L^2_{p,n}(X,F\otimes E,h_F\otimes h,\omega)$ satisfying $\int_X\langle B^{-1}_{h_F,\omega}f,f\rangle_{h_F\otimes h,\omega}dV_\omega<+\infty$, we have that 
    \begin{align*}
        \int_{S_j(k)}\langle[i\Theta_{F\otimes E,h_F\otimes h_\nu},\Lambda_{\omega}]^{-1}f,f\rangle_{h_F\otimes h_\nu,\omega} dV_{\omega}&\leq\int_{S_j(k)}\langle B^{-1}_{h_F,\omega}f,f\rangle_{h_F\otimes h_\nu,\omega} dV_{\omega}\\
        &\leq\int_X\langle B^{-1}_{h_F,\omega}f,f\rangle_{h_F\otimes h,\omega} dV_{\omega}<+\infty.
    \end{align*}

    By [39,\,Theorem\,3.7], 
    i.e. $L^2$-estimates for $(p,n)$-forms without completeness of $\omega$, 
    we get a solution $u_{j,k,\nu}\in L^2_{p,n-1}(S_j(k),F\otimes E,h_F\otimes h_\nu,\omega)$ of $\overline{\partial}u_{j,k,\nu}=f$ on $S_j(k)$ satisfying
    \begin{align*}
        \int_{S_j(k)}|u_{j,k,\nu}|^2_{h_F\otimes h_\nu,\omega}dV_\omega&\leq\int_{S_j(k)}\langle[i\Theta_{F\otimes E,h_F\otimes h_\nu},\Lambda_\omega]^{-1}f,f\rangle_{h_F\otimes h_\nu,\omega}dV_{\omega}\\
        &\leq\int_X\langle B^{-1}_{h_F,\omega}f,f\rangle_{h_F\otimes h,\omega} dV_{\omega}<+\infty.
    \end{align*}

    Here, $h_\nu$ increasing to $h$ a.e. as $\nu$ tends to $+\infty$.
    By repeating the discussion in the proof of $(b)$ of Theorem \ref{psef * Nak semi-posi then Nak semi-posi}, i.e. taking weak limit $\nu\to+\infty$ and $k\to+\infty$, 
    we get a solution $u_j\in L^2_{p,n-1}(S_j,F\otimes E,h_F\otimes h,\omega)$ of $\overline{\partial}u_j=f$ on $S_j$ satisfying
    \begin{align*}
        \int_{S_j}|u_j|^2_{h_F\otimes h,\omega}dV_{\omega}&\leq\int_{S_j}\langle B^{-1}_{h_F,\omega}f,f\rangle_{h_F\otimes h,\omega} dV_{\omega}=\int_X\langle B^{-1}_{h_F,\omega}f,f\rangle_{h_F\otimes h,\omega} dV_{\omega}<+\infty,
    \end{align*}
    where Lebesgue measures of $Z$ and $H_j$ are zero.
    From Lemma \ref{Ext d-equation for hypersurface}, letting $u_j=0$ on $Z$ and $H_j$ then we have that $u_j\in L^2_{p,n-1}(X_j,F\otimes E,h_F\otimes h,\omega)$ of $\overline{\partial}u_j=f$ on $X_j$ satisfying
    \begin{align*}
        \int_{X_j}|u_j|^2_{h_F\otimes h,\omega}dV_{\omega}&\leq\int_X\langle B^{-1}_{h_F,\omega}f,f\rangle_{h_F\otimes h,\omega} dV_{\omega}<+\infty.
    \end{align*}

    From repeating the same discussion as above, i.e. taking weak limit $j\to+\infty$, 
    we get a solution $u\in L^2_{p,n-1}(X,F\otimes E,h_F\otimes h,\omega)$ of $\overline{\partial}u=f$ on $X$ satisfying
    \begin{align*}
        \int_X|u_j|^2_{h_F\otimes h,\omega}dV_{\omega}&\leq\int_X\langle B^{-1}_{h_F,\omega}f,f\rangle_{h_F\otimes h,\omega} dV_{\omega}<+\infty.
    \end{align*}

    $(a)$ It is shown from the fact $A^{p,q}_{A,h_A,\omega}>0$ for $p+q>n+k-1$ as above.
\end{proof}

\section{An $L^2$-type Dolbeault isomorphism}

In this section, we provide $L^2$-type Dolbeault isomorphisms including $L^2$-subsheaves 
by using the following lemma and theorem.

\begin{lemma}\label{Regularity of currents for (p,0)-forms}$(\mathrm{Dolbeault}$-$\mathrm{Grothendieck \,lemma},\, \mathrm{cf.\,[8,\,ChapterI]})$ 
    Let $T$ be a current of type $(p,0)$ on some open subset $U\subset \mathbb{C}^n$. If $T$ is $\overline{\partial}$-closed then it is a holomorphic differential form, i.e. a smooth differential form with holomorphic coefficients.
\end{lemma}

\begin{theorem}\label{Ext Hor Thm to vect bdl}$(\mathrm{cf.\,[39,\, Theorem\, 6.1]})$
    Let $X$ be a complex manifold and $E$ be a holomorphic vector bundle equipped with a singular Hermitian metric $h$.
    We assume that $h$ is Griffiths semi-positive. Then for any $x_0\in X$, there exist an open neighborhood $U$ of $x_0$ and a \kah metric $\omega$ on $U$ satisfying that 
    for any $\overline{\partial}$-closed $f\in L^2_{p,q}(U,E\otimes\mathrm{det}\,E,h\otimes\mathrm{det}\,h,\omega)$, there exists $u\in L^2_{p,q-1}(U,E\otimes\mathrm{det}\,E,h\otimes\mathrm{det}\,h,\omega)$ such that $\overline{\partial}u=f$.
\end{theorem}

For singular Hermitian metrics $h$ on $E$, we define the subsheaf $\mathscr{L}^{p,q}_{E,h}$ of germs of $(p,q)$-forms $u$ with values in $E$ and with measurable coefficients such that both $|u|^2_{h}$ and $|\overline{\partial}u|^2_h$ are locally integrable, where $\mathscr{L}^{p,q}_{E,h}$ is a fine sheaf.

\begin{theorem}\label{L2-type Dolbeault isomorphism with sHm}
    Let $X$ be a complex manifold of dimension $n$ and $(F,h_F)$ be a holomorphic vector bundle over $X$. 
    Let $L$ be a holomorphic line bundle over $X$ equipped with a singular Hermitian metric $h_L$ and $E$ be a holomorphic vector bundle over $X$ equipped with a singular Hermitian metric $h_E$.
    Then we have the following
    \begin{itemize}
        \item [($a$)] If $h_L$ is singular semi-positive, then we have an exact sequence of sheaves 
        \begin{align*}
            0\longrightarrow \Omega_X^p\otimes \mathcal{O}_X(F\otimes L)\otimes \mathscr{I}(h_L) \longrightarrow \mathscr{L}^{p,\ast}_{F\otimes L,h_F\otimes h_L}.
        \end{align*}
        \item [($b$)] If $h_E$ is $L^2$-type Nakano semi-positive, then we get an exact sequence of sheaves 
        \begin{align*}
            0\longrightarrow K_X\otimes \mathcal{O}_X(F)\otimes \mathscr{E}(h_E) \longrightarrow \mathscr{L}^{n,\ast}_{F\otimes E,h_F\otimes h_E}.
        \end{align*}
        \item [($c$)] If $h_E$ is Griffiths semi-positive, then we have an exact sequence of sheaves 
        \begin{align*}
            0\longrightarrow \Omega_X^p\otimes \mathcal{O}_X(F)\otimes \mathscr{E}(h_E\otimes\mathrm{det}\,h_E) \longrightarrow \mathscr{L}^{p,\ast}_{F\otimes E\otimes\mathrm{det}\,E,h_F\otimes h_E\otimes\mathrm{det}\,h_E}.
        \end{align*}
    \end{itemize}
    In particular, $L^2$-type Dolbeault isomorphisms are obtained from these. For example, $H^q(X,\Omega_X^p\otimes F\otimes L\otimes \mathscr{I}(h_L))\cong H^q(\Gamma(X,\mathscr{L}^{p,\ast}_{F\otimes L,h_F\otimes h_L}))$ in the case of $(a)$.
\end{theorem}

To simplify the proof, we introduce the following definition. 

\begin{definition}
    Let $E$ be a holomorphic vector bundle on a complex manifold $X$. Consider two singular Hermitian metrics $h_1$ and $h_2$ on $E$. 
    For any open set $U$ of $X$, we will write $h_1\thicksim h_2$ on $U$, 
    if there is a constant $C>0$ such that $C^{-1}h_2\leq h_1\leq Ch_2$. 
\end{definition}

$\textit{Proof of Theorem \ref{L2-type Dolbeault isomorphism with sHm}}$.
    For any fixed point $x_0\in X$, there exist a Stein open neighborhood $U$ of $x_0$ such that $F$ is trivial on $U$, i.e. $F|_U=\mathbb{C}^r\times U:=\underline{\mathbb{C}^r}$, where $r=\mathrm{rank}\,F$. 
    Let $(U;z_1,\cdots,z_n)$ be a local coordinate, $I_F$ be a trivial Hermitian metric on $F|_U$ and $\omega=\sum^n_{j=1} dz_j\wedge d\overline{z}_j$ be a \kah metric. 
    By smoothness of $h_F$, we get $h_F\thicksim I_{F}$ on $U$.

    We consider the following sheaves sequence: 
    \begin{align*}
        \xymatrix{
        0 \ar[r] & \mathrm{ker}\,\overline{\partial}_0 \hookrightarrow \mathscr{L}^{p,0}_{F\otimes E,h_F\otimes h_E} \ar[r]^-{\overline{\partial}_0} & \mathscr{L}^{p,1}_{F\otimes E,h_F\otimes h_E} \ar[r]^-{\overline{\partial}_1} & \cdots \ar[r]^-{\overline{\partial}_{n-1}} & \mathscr{L}^{p,n}_{F\otimes E,h_F\otimes h_E} \ar[r] & 0,
    }
    \end{align*}
    where $\overline{\partial}_j=\overline{\partial}_{F\otimes E}=\overline{\partial}\otimes\mathrm{id}_{F\otimes E}$.
    From $h_F\thicksim I_F$ on $U$, we get $h_F\otimes h_E\thicksim I_F\otimes h_E$ on $U$ and $L^2_{p,q}(U,F\otimes E,h_F\otimes h_E,\omega)=L^2_{p,q}(U,\underline{\mathbb{C}^r}\otimes E,I_F\otimes h_E,\omega)$. Therefore, $\mathscr{L}^{p,q}_{F\otimes E,h_F\otimes h_E}(U)=\mathscr{L}^{p,q}_{\underline{\mathbb{C}^r}\otimes E,I_F\otimes h_E}(U)$.
    Since Lemma \ref{Regularity of currents for (p,0)-forms}, the kernel of $\overline{\partial}_0$ consists of all germs of holomorphic $(p,0)$-forms with values in $F\otimes E$ which satisfy the integrability condition. 
    
    We prove that $\mathrm{ker}\,\overline{\partial}_0=\Omega^p_X\otimes \mathcal{O}_X(F)\otimes \mathscr{E}(h_E)$. 
    Let $e=(e_1,\cdots,e_r)$ and $b=(b_1,\cdots,b_s)$ be holomorphic frames of $\underline{\mathbb{C}^r}$ and $F$ on $U$ respectively, where $s=\mathrm{rank}\,E$, $e_j=(0,\cdots,0,1,0,\cdots,0)$ and $e$ is orthonormal with respect to $I_F$.
    For any $f\in H^0(U,\Omega^p_X\otimes \underline{\mathbb{C}^r}\otimes E)=H^0(U,\Omega^p_X\otimes F\otimes E)$, we can write 
    \begin{align*}
        f=\sum_{|I|=p,j,\lambda} f_{Ij\lambda}dz_I\otimes e_j\otimes b_\lambda=\sum_j f_j\otimes e_j=\sum_{|I|=p,j}f_{jI}dz_I\otimes e_j,
    \end{align*}
    where $f_j=\sum_{|I|=p,\lambda} f_{Ij\lambda}dz_I\otimes b_\lambda=\sum_{|I|=p} f_{jI} dz_I \in H^0(U,\Omega^p_X\otimes E)$ and $f_{jI}=\sum_\lambda f_{Ij\lambda}\otimes b_\lambda \in H^0(U,E)$.
    We can calculate the following
    \begin{align*}
        |f|^2_{I_F\otimes h_E,\omega}=\sum_j|f_j|^2_{h_E,\omega}=\sum_{j,I}|f_{jI}|^2_{h_E}, \quad
        ||f||^2_{I_F\otimes h_E,\omega}
        =\sum_{j,I}\int_U|f_{jI}|^2_{h_E}dV_\omega.
    \end{align*}

    Therefore, we get $f\in\mathrm{ker}\,\overline{\partial}_0(U) \iff ||f||^2_{I_F\otimes h_E,\omega}=\sum_{j,I}\int_U|f_{jI}|^2_{h_E}dV_\omega<+\infty,$
    i.e. for any $f_{jI}\in H^0(U,E)$ satisfy the condition $f_{jI}\in\mathscr{E}(h_E)(U)$. Hence, we have that $\mathrm{ker}\,\overline{\partial}_0=\Omega^p_X\otimes \mathcal{O}_X(F)\otimes \mathscr{E}(h_E)$.
    In particular, from the fact $\mathscr{E}(h_L)=\mathcal{O}_X(L)\otimes\mathscr{I}(h_L)$, if $(E,h_E)=(L,h_L)$ then we obtain $\mathrm{ker}\,\overline{\partial}_0=\Omega^p_X\otimes \mathcal{O}_X(F\otimes L)\otimes \mathscr{I}(h_L)$.

    From the above, the sheaves sequences of $(a)$-$(c)$ are exact at $q=0$. Finally, we prove the exactness of the sheaves sequences of $(a)$-$(c)$ at $q\geq1$.

    $(a)$ We can retake $U$ such that $L$ is also trivial on $U$.\! By the assumption, $\varphi\!:=\!-\!\log h_L$ is plurisubharmonic on $U$.
    Since $\mathscr{L}^{p,q}_{F\otimes L,h_F\otimes h_L}(U)=\mathscr{L}^{p,q}_{\underline{\mathbb{C}^r}\otimes L,I_F\otimes h_L}(U)$, it is sufficient to show that for any $\overline{\partial}$-closed $f\in L^2_{p,q}(U,\underline{\mathbb{C}^r}\otimes L,I_F\otimes h_L,\omega)$, there is $u\in L^2_{p,q-1}(U,\underline{\mathbb{C}^r}\otimes L,I_F\otimes h_L,\omega)$ such that $\overline{\partial}u=f$.
    Let $\widetilde{e}$ be a holomorphic frame of $L|_U$. We can write 
    \begin{align*}
        f=\sum_{|I|=p,|J|=q,k}f_{IJk}dz_I\wedge d\overline{z}_J\otimes e_k\otimes \widetilde{e}=\sum_k f_k \otimes e_k,\quad f_k=\sum_{|I|=p,|J|=q}f_{IJk}dz_I\wedge d\overline{z}_J\otimes \widetilde{e}.
    \end{align*}
    Here, $||f||^2_{I_F\otimes h_L,\omega}=\sum_k\int_U|f_k|^2 e^{-\varphi}dV_\omega<+\infty$, 
    i.e. $f_k \in L^2_{p,q}(U,L,h_L,\omega)=L^2_{p,q}(U,\varphi,\omega)$ for any $k$. From holomorphicity of $e$, we get 
    $0=\overline{\partial}f=\overline{\partial}\sum_k f_k \otimes e_k=\sum_k \overline{\partial}f_k \otimes e_k$.
    
    Thus, $\overline{\partial}f_k=0$. 
    By [15, Theorem\,4.4.2], there is a solution $u_k$ of $\overline{\partial}u_k=f_k$ satisfying
    \begin{align*}
        \inf_U(1+|z|^2)^{-2}\int_U|u_k|^2 e^{-\varphi}dV_\omega\leq\int_U|u_k|^2 e^{-\varphi}(1+|z|^2)^{-2}dV_\omega\leq\int_U|f_k|^2 e^{-\varphi}dV_\omega,
    \end{align*}
    where $\inf_U(1+|z|^2)^{-2}>0$.
    We define the $(p,q-1)$-forms $u=\sum_ku_k\otimes e_k$, then we get
    \begin{align*}
        &\overline{\partial}u=\overline{\partial}\sum u_k\otimes e_k=\sum \overline{\partial}u_k\otimes e_k=\sum f_k\otimes e_k=f,\\    
        &\inf_U(1+|z|^2)^{-2}||u||^2_{I_F\otimes h_L,\omega}
        =\inf_U(1+|z|^2)^{-2}\sum\int_U|u_k|^2 e^{-\varphi}dV_\omega\\
        &\leq\sum\int_U|u_k|^2 e^{-\varphi}(1+|z|^2)^{-2}dV_\omega\leq\sum\int_U|f_k|^2 e^{-\varphi}dV_\omega<+\infty.
    \end{align*}

    $(b)$ Let $\psi:=|z|^2$ be a smooth strictly plurisubharmonic on $U$ such that $\idd\psi=\omega$. From $h_F\thicksim I_F\thicksim I_Fe^{-\psi}:=I_F^\psi$ on $U$, we get $\mathscr{L}^{n,q}_{F\otimes E,h_F\otimes h_E}(U)=\mathscr{L}^{n,q}_{\underline{\mathbb{C}^r}\otimes E,I_F^\psi\otimes h_E}(U)$. 
    Then, it is sufficient to show that for any $\overline{\partial}$-closed $f\in L^2_{n,q}(U,\underline{\mathbb{C}^r}\otimes E,I_F^\psi\otimes h_E,\omega)$, there exists $u\in L^2_{n,q-1}(U,\underline{\mathbb{C}^r}\otimes E,I_F^\psi\otimes h_E,\omega)$ such that $\overline{\partial}u=f$.
    We can write 
    \begin{align*}
        f=\sum f_{Jk\lambda}dz_N\wedge d\overline{z}_J\otimes e_k\otimes b_\lambda=\sum f_k \otimes e_k, \,\,f_k=\sum f_{Jk\lambda}dz_N\wedge d\overline{z}_J\otimes b_\lambda.
    \end{align*}
    
    Then we get $f_k \in L^2_{n,q}(U,E,h_Ee^{-\psi},\omega)$ and $\overline{\partial}f_k=0$ on $U$ for any $k$. 
    Here, we obtain $B_{\psi,\omega}=q$ on $\Lambda^{n,q}T^*_U\otimes E$.
    By $L^2$-type Nakano semi-positivity of $h_E$, for any $\overline{\partial}$-closed $f_k\in L^2_{n,q}(U,E,h_Ee^{-\psi},\omega)$, there is $u_k\in L^2_{n,q-1}(U,E,h_Ee^{-\psi},\omega)$ satisfying $\overline{\partial}u_k=f$ and 
    \begin{align*}
        \int_U|u_k|^2_{h_E,\omega}e^{-\psi}dV_\omega\leq\int_U\langle B_{\psi,\omega}^{-1}f_k,f_k\rangle_{h_E,\omega}e^{-\psi}dV_\omega=\frac{1}{q}\int_U|f_k|^2_{h_E,\omega}e^{-\psi}dV_\omega<+\infty.
    \end{align*}
    We define the $(n,q-1)$-form $u:=\sum_ku_k\otimes e_k$. Then we have that $\overline{\partial}u=f$ and
    \begin{align*}
        ||u||^2_{I_F^\psi\otimes h_E,\omega}&=\int_U|u|^2_{I_F^\psi\otimes h_E,\omega}dV_\omega=\sum_k\int_U|u_k|^2_{h_E,\omega}e^{-\psi}dV_\omega<+\infty.
    \end{align*}

    $(c)$ It is shown using Theorem \ref{Ext Hor Thm to vect bdl} in the same way as $(a)$. \qed

\section{Main results and proofs for vanishing theorems}

In this section, we prove main results and additionally give vanishing theorems for the cases where a singular Hermitian metric is $L^2$-type Nakano semi-positive and dual Nakano semi-positive. 

\vspace{2mm}

$\textit{Proof of Theorem \ref{Ext NDBS V-thm for psef for (n,q)}}.$
    Let $h_F$ be a smooth Hermitian metric on $F$. By $(a)$ of Theorem \ref{L2-type Dolbeault isomorphism with sHm}, the complex of sheaves $(\mathscr{L}^{p,\ast}_{F\otimes L,h_F\otimes h},\overline{\partial})$ defined by $\overline{\partial}$-operator is a resolution of the sheaf $\Omega_X^p\otimes\mathcal{O}_X(F\otimes L)\otimes\mathscr{I}(h)$.
    Since acyclicity of each $\mathscr{L}^{p,q}_{F\otimes L,h_F\otimes h}$, we have that 
    \begin{align*}
        H^q(X,\Omega_X^p\otimes F\otimes L\otimes \mathscr{I}(h))\cong H^q(\Gamma(X,\mathscr{L}^{p,\ast}_{F\otimes L,h_F\otimes h})).
    \end{align*} 
    Let $\psi$ be a smooth exhaustive plurisubharmonic function on $X$. For any convex increasing function $\chi\in\mathcal{C}^{\infty}(\mathbb{R},\mathbb{R})$, we define the smooth Hermitian metric $h_F^\chi:=h_Fe^{-\chi\circ\psi}$.

    $(b)$ We assume $m$-positivity of $h_F$. Then $h_F^\chi$ is also $m$-positive, $h_F$ is Griffiths positive and $(\mathrm{det}\,F,\mathrm{det}\,h_F)$ is positive. Thus, there exists a \kah metric $\omega$ on $X$. 
    
    From Lemma \ref{m-posi then A>0 and dual m-posi then A>0}, for any positive integer $q\geq1$ with $m\geq\min\{n-q+1,r\}$ we have that $A^{n,q}_{F,h_F,\omega}>0$ and $A^{n,q}_{F,h_F^\chi,\omega}>0$.
    Moreover, any $(n,q)$-form $v\in\Lambda^{n,q}T^*_X\otimes F\otimes L$, we get $\langle B_{h_F,\omega}v,v\rangle_{h_F\otimes h,\omega}>0$ and $\langle B_{h_F^\chi,\omega}v,v\rangle_{h_F^\chi\otimes h,\omega}>0$ by the proof of Proposition \ref{A_(E+F)>C_E+C_F if A_E>C_E and A_F>C_F}, where $B_{h_F,\omega}=[i\Theta_{F,h_F}\otimes\mathrm{id}_L,\Lambda_\omega], B_{h_F^\chi,\omega}=[i\Theta_{F,h_F^\chi}\otimes\mathrm{id}_L,\Lambda_\omega]$.
    Then, from the inequality $\langle B_{h_F^\chi,\omega}v,v\rangle_{h_F^\chi\otimes h,\omega}\geq\langle B_{h_F,\omega} v,v\rangle_{h_F^\chi\otimes h,\omega}=\langle B_{h_F,\omega} v,v\rangle_{h_F\otimes h,\omega}e^{-\chi\circ\psi}>0$,
    we have that
    \begin{align*}
        0<\langle B_{h_F^\chi,\omega}^{-1}v,v\rangle_{h_F^\chi\otimes h,\omega}\leq\langle B_{h_F,\omega}^{-1}v,v\rangle_{h_F\otimes h,\omega}e^{-\chi\circ\psi}.
    \end{align*}

    In fact, for any $(n,q)$-forms $u,v\in \Lambda^{n,q} T^*_X\otimes F\otimes L$, we get
    \begin{align*}
        |\langle v,u\rangle_{h_F^\chi\otimes h,\omega}|^2=|\langle v,u\rangle_{h_F\otimes h,\omega}|^2e^{-2\chi\circ\psi}&\leq\langle B_{h_F,\omega}^{-1}v,v\rangle_{h_F\otimes h,\omega}\langle B_{h_F,\omega}u,u\rangle_{h_F\otimes h,\omega}e^{-2\chi\circ\psi}\\
        &\leq\langle B_{h_F,\omega}^{-1}v,v\rangle_{h_F\otimes h,\omega}\langle B_{h_F^\chi,\omega}u,u\rangle_{h_F^\chi\otimes h,\omega}e^{-\chi\circ\psi},
    \end{align*}
    and the choose $u=B_{h_F^\chi,\omega}^{-1}v$ implies 
    \begin{align*}
        |\langle v,B_{h_F^\chi,\omega}^{-1}v\rangle_{h_F^\chi\otimes h,\omega}|^2\leq \langle B_{h_F,\omega}^{-1}v,v\rangle_{h_F\otimes h,\omega}e^{-\chi\circ\psi} \cdot\langle v,B_{h_F^\chi,\omega}^{-1}v\rangle_{h_F^\chi\otimes h,\omega}.
    \end{align*}

    For any $\overline{\partial}$-closed $f\in \Gamma(X,\mathscr{L}^{n,q}_{F\otimes L,h_F\otimes h})$, 
    the integrals
    \begin{align*}
        \int_X|f|^2_{h_F^\chi\otimes h,\omega}dV_\omega&=\int_X|f|^2_{h_F\otimes h,\omega}e^{-\chi\circ\psi}dV_\omega \quad \mathrm{and}\\
        \int_X\langle B_{h_F^\chi,\omega}^{-1}f,f\rangle_{h_F^\chi\otimes h,\omega}dV_\omega&\leq\int_X\langle B_{h_F,\omega}^{-1}f,f\rangle_{h_F\otimes h,\omega}e^{-\chi\circ\psi}dV_\omega
    \end{align*}
    become convergent if $\chi$ grows fast enough. By Theorem \ref{L2-estimate of m-posi + psef}, there exists $u\in L^2_{n,q-1}(X,F\otimes L,h_F^\chi\otimes h,\omega)$ such that $\overline{\partial}u=f$ and 
    \[
        \int_X|u|^2_{h_F\otimes h,\omega}e^{-\chi\circ\psi}dV_\omega\leq\int_X\langle B_{h_F^\chi,\omega}^{-1}f,f\rangle_{h_F\otimes h,\omega}e^{-\chi\circ\psi}dV_\omega<+\infty,
    \]
    where $|u|^2_{h_F\otimes h,\omega}$ is locally integrable.
    Hence, we have that $u\in \Gamma(X,\mathscr{L}^{n,q-1}_{F\otimes L,h_F\otimes h})$ and that $H^q(X,K_X\otimes F\otimes L\otimes \mathscr{I}(h))=0$.

    $(a)$ This is shown in the same way as above using Corollary \ref{L2-estimate of k-posi + psef}. \qed\\

    Similarly, Theorem \ref{Ext NDBS V-thm for psef for (p,n)} can be shown by using Theorem \ref{L2-estimate of m-posi + psef} and Corollary \ref{L2-estimate of k-posi + psef}, and Theorem \ref{Ext NDBS V-thm for big for (n,q)} and \ref{Ext NDBS V-thm for big for (p,n)} can be shown by using Theorem \ref{L2-type Dolbeault isomorphism with sHm} and \ref{L2-estimate of m-semi-posi + big}.
    Furthermore, we obtain the following theorem and corollary for $L^2$-type Nakano semi-positive similarly to the above using $(b)$ of Theorem \ref{L2-type Dolbeault isomorphism with sHm} and Theorem \ref{L2-estimates of L2-type Nakano semi-posi + m-posi on w.p.c}.

\begin{theorem}\label{V-thm of L2-type Nakano semi-posi + m-posi}
    Let $X$ be a weakly pseudoconvex manifold 
    and $E$ be a holomorphic vector bundle equipped with a singular Hermitian metric $h$ which is $L^2$-type Nakano semi-positive on $X$. 
    Then we have the following 
    \begin{itemize}
        \item [$(a)$] If $X$ has a positive holomorphic line bundle and $A$ is a $k$-positive line bundle, then for any $q\geq k$ we have that
        \begin{align*}
            H^q(X,K_X\otimes A\otimes \mathscr{E}(h))=0.
        \end{align*}
        \item [$(b)$] If $F$ is a $m$-positive holomorphic vector bundle of rank $r$ then 
        \begin{align*}
            H^q(X,K_X\otimes F\otimes \mathscr{E}(h))=0
        \end{align*}
        for $q\geq1$ with $m\geq\min\{n-q+1,r\}$.
    \end{itemize}
\end{theorem}

\begin{corollary}\label{V-thm of L2-type Nakano posi on w.p.c}
    Let $X$ be a weakly pseudoconvex manifold and $E$ be a holomorphic vector bundle equipped with a singular Hermitian metric $h$. 
    We assume that there exists a holomorphic positive line bundle $(L,h_L)$ such that the singular Hermitian metric $h\otimes h_L^*$ on $E\otimes L^*$ is $L^2$-type Nakano semi-positive on $X$. 
    Then we have the following vanishing
    \begin{align*}
        H^q(X,K_X\otimes\mathscr{E}(h))=0
    \end{align*}
    for any $q>0$.
\end{corollary}

Theorem \ref{V-thm of Griffiths semi-posi + (dual)-m-posi for (n,q)} is proved using Theorem \ref{V-thm of L2-type Nakano semi-posi + m-posi} and [19,\,Theorem\,1.3], i.e. if $h$ is Griffiths semi-positive then $h\otimes\mathrm{det}\,h$ is $L^2$-type Nakano semi-positive. 
The following 
theorem for strictly dual Nakano positivity 
is obtained, 
which is generalized from Hodge to \kah using the results of $\S 3.2$, and allows for more singularity then in [39,\,Theorem\,1.2]. 

\begin{theorem}\label{V-thm of s-dual Nak d-posi on proj alg}
    Let $X$ be a projective manifold equipped with a \kah metric $\omega$. 
    Let $E$ be a holomorphic vector bundle over $X$ equipped with a singular Hermitian metric $h$.
    We assume that $h$ is strictly dual Nakano $\delta_\omega$-positive on $X$ and that $\nu(-\log\mathrm{det}\,h,x)<2$ for any point $x\in X$.
    Then for any $p>0$, we have the cohomology vanishing
    \begin{align*}
        H^n(X,\Omega_X^p\otimes E)=0.
    \end{align*}
\end{theorem}

\begin{proof}
    Let $\mathcal{E}^{p,q}(X,E)$ be the space of smooth $E$-valued $(p,q)$-forms on $X$ and $\mathcal{U}=\{U_j\}_{j\in I}$ be a locally finite open cover of $X$ such that $U_j$ are biholomorphic to a polydisc.
    By the assumption, $\mathrm{det}\,h$ is locally integrable from the results of Skoda (see \cite{Sko72}). Since $h=\mathrm{det}\,h\cdot \widehat{h^*}$ and each element of $\widehat{h^*}$ is locally bounded (see [31,\,Lemma\,2.2.4]), for any $s\in\mathcal{E}^{p,q}(X,E)$ the function $|s|^2_h$ is also locally integrable. 
    Here, $\widehat{h^*}$ is the adjugate matrix of $h^*$. Thus, there is an inclusion map $\mathcal{E}^{p,q}(X,E)\hookrightarrow L^2_{\mathrm{loc}(p,q)}(X,E,h,\omega)$.

    We know that $U_{j_0}\cap\cdots\cap U_{j_l}$ is a pseudoconvex domain for all $\{j_0,\ldots,j_l\}\subset I$.
    Since [39,\,Theorem\,4.12], we can solve the $\overline{\partial}$-equation on $U_{j_0}\cap\cdots\cap U_{j_l}$ with respect to $h$. 

    Hence, we have the isomorphism $H^n(X,\Omega_X^p\otimes E)$
    \begin{align*}
        \cong\frac{\{f\in L^2_{\mathrm{loc}(p,n)}(X,E,h,\omega);\overline{\partial}f=0\}}{\{g\in L^2_{\mathrm{loc}(p,n)}(X,E,h,\omega);\mathrm{there ~is~ an~} \gamma\in L^2_{\mathrm{loc}(p,n-1)}(X,E,h,\omega) \,\mathrm{satisfying}\, \overline{\partial}\gamma=g\}}
    \end{align*}
    from the results of sheaf cohomology. This is a singular version of isomorphism theorems (see \cite{Ohs82}) and was first mentioned in [18,\,Corollary\,1.2]. 
    By the projectivity of $X$ and Proposition \ref{s-positivity for two metric}, $h$ has strictly dual Nakano positivity for a Hodge metric on $X$.
    Therefore, we obtain $H^n(X,\Omega_X^p\otimes E)=0$ from [39,\,Theorem\,4.12]. 
\end{proof}

Using Theorem \ref{L2-estimates of dual Nakano semi-posi + m-posi on w.p.c} and the proof method of Theorem \ref{V-thm of s-dual Nak d-posi on proj alg}, we obtain the following theorem for dual Nakano semi-positivity.

\begin{theorem}\label{V-thm of dual Nakano semi-posi + dual m-posi}
    Let $X$ be a projective manifold, $F$ be a holomorphic vector bundle of rank $r$ and $E$ be a holomorphic vector bundle equipped with a singular Hermitian metric $h$.
    We assume that $h$ is dual Nakano semi-positive on $X$ satisfying $\nu(-\log\mathrm{det}\,h,x)<2$ for any point $x\in X$.
    Then we have the following
    \begin{itemize}
        \item [$(a)$] If $A$ is a $k$-positive line bundle then, for any $p\geq k$ we have that
        \begin{align*}
            H^n(X,\Omega^p_X\otimes A\otimes E)=0.
        \end{align*}
        \item [$(b)$] If $F$ is a dual $m$-positive holomorphic vector bundle of rank $r$ then 
        \begin{align*}
            H^n(X,\Omega^p_X\otimes F\otimes E)=0
        \end{align*}
        for $p\geq1$ with $m\geq\min\{n-p+1,r\}$.
    \end{itemize}
\end{theorem}

Applying Theorem \ref{Grif * det Grif is dual Nakano} and \ref{L2-estimates of dual Nakano semi-posi + m-posi on w.p.c} and $(c)$ of Theorem \ref{L2-type Dolbeault isomorphism with sHm}, we can prove Theorem \ref{V-thm of Griffiths semi-posi + (dual)-m-posi for (p,n)}.

\vspace{2mm}

$\textit{Proof of Theorem \ref{V-thm of Griffiths semi-posi + (dual)-m-posi for (p,n)}}.$
    Let $h_F$ be a smooth Hermitian metric on $F$. By $(c)$ of Theorem \ref{L2-type Dolbeault isomorphism with sHm}, 
    we have that $H^q(X,\Omega_X^p\otimes F\otimes \mathscr{E}(h\otimes\mathrm{det}\,h))\cong H^q(\Gamma(X,\mathscr{L}^{p,\ast}_{F\otimes E\otimes\mathrm{det}\,E,h_F\otimes h\otimes\mathrm{det}\,h}))$.

    $(b)$ Let $\omega$ be a \kah metric on $X$. We assume dual $m$-positivity of $h_F$. 
    From Lemma \ref{m-posi then A>0 and dual m-posi then A>0}, for any positive integer $p\geq1$ with $m\geq\min\{n-p+1,r\}$ we have that $A^{p,n}_{F,h_F,\omega}>0$.
    Moreover, any $(n,q)$-form $v\in\Lambda^{n,q}T^*_X\otimes F\otimes E\otimes\mathrm{det}\,E$ we get $\langle B_{h_F,\omega}v,v\rangle_{h_F\otimes h\otimes\mathrm{det}\,h,\omega}>0$ by the proof of Proposition \ref{A_(E+F)>C_E+C_F if A_E>C_E and A_F>C_F}, where $B_{h_F,\omega}=[i\Theta_{F,h_F}\otimes\mathrm{id}_{E\otimes\mathrm{det}\,E},\Lambda_\omega]$.

    By compactness of $X$, for any $\overline{\partial}$-closed $f\in \Gamma(X,\mathscr{L}^{n,q}_{F\otimes E\otimes\mathrm{det}\,E,h_F\otimes h\otimes\mathrm{det}\,h})$ satisfies $\int_X\langle B_{h_F,\omega}^{-1}f,f\rangle_{h_F\otimes h\otimes\mathrm{det}\,h,\omega}dV_\omega<+\infty$.
    Since Theorem \ref{Grif * det Grif is dual Nakano} and \ref{L2-estimates of dual Nakano semi-posi + m-posi on w.p.c}, there exists $u\in L^2_{p,n-1}(X,F\otimes E\otimes\mathrm{det}\,E,h_F\otimes h\otimes\mathrm{det}\,h,\omega)$ such that $\overline{\partial}u=f$ and 
    \begin{align*}
        \int_X|u|^2_{h_F\otimes h\otimes\mathrm{det}\,h,\omega}dV_\omega\leq\int_X\langle B_{h_F,\omega}^{-1}f,f\rangle_{h_F\otimes h\otimes\mathrm{det}\,h,\omega}dV_\omega<+\infty,
    \end{align*}
    where $|u|^2_{h_F\otimes h\otimes\mathrm{det}\,h,\omega}$ is locally integrable.
    Hence, we have that $u\in \Gamma(X,\mathscr{L}^{p,n-1}_{F\otimes E\otimes\mathrm{det}\,E,h_F\otimes h\otimes\mathrm{det}\,h})$ and that $H^n(X,\Omega^p_X\otimes F\otimes \mathscr{E}(h\otimes\mathrm{det}\,h))=0$.

    $(a)$ This is shown as above using the fact $A^{p,q}_{A,h_A,\omega}>0$ for $p+q>n+k-1$. 
\qed

\section{Fujita's conjecture type theorem with singular Hermitian metrics}

In \cite{Fuj88}, Fujita proposed the following conjecture which is a open question in classical algebraic geometry.
Recall that, $X$ is an $n$-dimensional complex manifold.

\begin{conjecture}
    Let $X$ be a smooth projective variety and $L$ be an ample line bundle. 
    \begin{itemize}
        \item $K_X\otimes L^{\otimes(\mathrm{dim}\,X+1)}$ is globally generated;
        \item $K_X\otimes L^{\otimes(\mathrm{dim}\,X+2)}$ is very ample.
    \end{itemize}
\end{conjecture}

The global generation conjecture has been proved (cf. [12,\,22,\,40]) 
up to dimension $5$.
Recently, Fujita's conjecture type theorems was obtained in \cite{SY19} for the case of pseudo-effective involving the multiplier ideal sheaf and for the case of nef involving Nakano semi-positive vector bundles, as follows.

\begin{theorem}\label{Fujita Conj thm for psef}$(\mathrm{cf.\,[35,\,Theorem\,1.3]})$ 
    Let $X$ be a compact \kah manifold, $L$ be an ample and globally generated line bundle and $(B,h)$ be a pseudo-effective line bundle.
    If the numerical dimension of $(B,h)$ is not zero, i.e. $\mathrm{nd}(B,h)\ne0$. then
    \begin{align*}
        K_X\otimes L^{\otimes n}\otimes B\otimes \mathscr{I}(h)
    \end{align*}
    is globally generated.
\end{theorem}

\begin{theorem}$(\mathrm{cf.\,[35,\,Theorem\,1.4]})$ 
    Let $X$ be a compact \kah manifold and $L$ be an ample and globally generated line bundle. 
    Let $E$ be a holomorphic vector bundle which is Nakano semi-positive. 
    If $N$ is a nef but not numerically trivial line bundle, then the adjoint vector bundle $K_X\otimes L^{\otimes n}\otimes N\otimes E$ is globally generated.
\end{theorem}

Here, Theorem \ref{Fujita Conj thm for psef} holds with the addition of a Nakano semi-positive vector bundle (see [35,\,Theorem\,4.3]). 
In this section, as an extension of these theorems to singular Hermitian metric of holomorphic vector bundles, we show Theorem \ref{Fujita Conj for Grif} involving the $L^2$-subsheaf with respect to $L^2$-type Nakano semi-positive singular Hermitian metric.

We introduce a concept on numerical dimension for nef line bundles.

\begin{definition}$(\mathrm{cf.\,[7,\,Definition\,6.20]})$ 
    Let $X$ be a compact \kah manifold of dimension $n$ and $N$ be a nef line bundle over $X$. 
    The $\it{numerical}$ $\it{dimension}$ $\nu(N)$ of $N$ is defined as $\nu(N)=\max\{k=0,\ldots,n\mid c^k_1(N)\ne0 ~\mathrm{in}~ H^{2k}(X,\mathbb{R})\}$.
\end{definition}

These global generation conjecture type theorems are shown using the theory of Castelnuovo-Mumford regularity and vanishing theorems. 

\begin{definition}$(\mathrm{cf.\,[24,\,Definition\,1.8.4]})$ 
    Let $X$ be a projective manifold and $L$ be an ample and globally generated line bundle over $X$.
    A coherent sheaf $\mathcal{F}$ on $X$ is $m$-$\it{regular}$ $\it{with}$ $\it{respect}$ $\it{to}$ $L$ if $H^q(X,\mathcal{F}\otimes L^{\otimes(n-q)})=0$ for $q>0$.
\end{definition}

\begin{lemma}\label{Castelnuovo-Mumford regularity}$(\mathrm{Mumford,~cf.\,[24,\,Theorem\,1.8.5]})$ 
    Let $\mathcal{F}$ be a $0$-regular coherent sheaf on $X$ with respect to $L$, then $\mathcal{F}$ is generated by its global sections.
\end{lemma}

To prove Theorem \ref{Fujita Conj for Grif}, we show the following vanishing theorem.

\begin{theorem}\label{V-thm of nu(N) and Grif for nu(h)<1 near A in 7}$(=\mathrm{Theorem}\,\ref{V-thm of nu(N) and Grif for nu(h)<1 near A})$
    Let $X$ be a compact \kah manifold of dimension $n$ and $E$ be a holomorphic vector bundle equipped with a singular Hermitian metric $h$. 
    Let $N$ be a nef line bundle which is neither big nor numerically trivial, i.e. $\nu(N)\notin\{0,n\}$. 
    If $h$ is Griffiths semi-positive and there exists a smooth ample divisor $A$ such that $\nu(-\log\mathrm{det}\,h|_A,x)<1$ for all points in $A$ and that $\nu(N|_A)=\nu(N)$, then we have that 
    \begin{align*}
        H^q(X,K_X\otimes N\otimes \mathscr{E}(h\otimes\mathrm{det}\,h))=0 \quad \mathit{for} \quad q>n-\nu(N).
    \end{align*}
\end{theorem}

We first consider the case where the condition for the Lelong number is the whole X using the following proposition.
This proposition is an example of when the equality of the subadditivity property (see [7,\,Theorem\,15.2]) to the $L^2$-subsheaf holds. 

\begin{proposition}\label{m.m.sheaf for Grif * nef big}
    Let $X$ be a projective manifold and $L$ be a nef and big line bundle.
    Let $E$ be a holomorphic vector bundle equipped with a singular Hermitian metric $h$.
    If $h$ is Griffiths semi-positive and $\nu(-\log \mathrm{det}\,h,x)<2$ for all points $x\in X$.
    Then there exists a singular Hermitian metric $h_L$ on $L$ such that $\mathscr{E}(h\otimes h_L)\cong\mathcal{O}_X(E\otimes L)$.
\end{proposition}

\begin{proof} 
    From Griffiths semi-positivity of $h$, $(\mathrm{det}\,E,\mathrm{det}\,h)$ is a pseudo-effective line bundle. By the assumption and Skoda's result \cite{Sko72}, the function $\mathrm{det}\,h$ is locally integrable, i.e. $1\in\mathscr{I}(\mathrm{det}\,h)_x$ for all points $x\in X$. 
    In other words, there exists $R>0$ such that $\int_{\mathbb{B}^n_R}\mathrm{det}\,h\,dV_{\mathbb{C}^n}<+\infty$,
    where $\mathbb{B}^n_R=\{z\in \mathbb{C}^n\mid |z|<R\}$ and $(z_1,\cdots,z_n)$ is a local coordinate around $x$.
    Since the strongly openness property (see. [14,\,20]), 
    for some $r\in(0,R)$ there exists $\beta_x>0$ such that $\int_{\mathbb{B}^n_r}(\mathrm{det}\,h)^{1+\beta_x}dV_{\mathbb{C}^n}<+\infty$.

    By the H\"older inequality, for any singular Hermitian metric $h_L$ on $L$ we get 
    \begin{align*}
        \int_{\mathbb{B}^n_r}\mathrm{det}\,h\cdot h_LdV_{\mathbb{C}^n}\leq\Bigl(\int_{\mathbb{B}^n_r}(\mathrm{det}\,h)^{1+\beta_x}dV_{\mathbb{C}^n}\Bigr)^{\frac{1}{1+\beta_x}}\Bigl(\int_{\mathbb{B}^n_r}h_L^{1+1/\beta_x}dV_{\mathbb{C}^n}\Bigr)^\frac{\beta_x}{1+\beta_x}.
    \end{align*}

    Since $L$ is nef and big, for every $\delta>0$, $L$ has a singular Hermitian metric $h_L$ such that $\max_{x\in X}\nu(-\log h_L,x)<\delta$ and $i\Theta_{L,h_L}\geq\varepsilon\omega$ for some $\varepsilon>0$ (see [7,\,Corollary\,6.19]), 
    where $\omega$ is a \kah metric.
    Let $\beta=\min_{x\in X}\beta_x>0$ and $\delta=2\beta/(1+\beta)$ then for any point $x\in X$, we get $\nu(-\log h_L,x)<\delta\leq2\beta_x/(1+\beta_x)$, i.e. $\nu(-\log h_L^{1+1/\beta_x},x)<2$.
    Therefore, $h_L^{1+1/\beta_x}$ is locally integrable at $x$ and $\mathrm{det}\,h\cdot h_L$ is also locally integrable. 

    From $h=\mathrm{det}\,h\cdot \widehat{h^*}$ and each element of $\widehat{h^*}$ is locally bounded [31,\,Lemma\,2.2.4], 
    for any local section $s\in C^{\infty}_x(E\otimes L)$ the function $|s|^2_{h\otimes h_L}$ is locally integrable. Here, $\widehat{h^*}$ is the adjugate matrix of $h^*$.
    Hence, the proof is complete from Definition \ref{def of multiplier submodule}.
\end{proof}

\begin{corollary}\label{m.m.sheaf for Grif * det Grif * nef big}
    Let $X$ be a projective manifold and $L$ be a nef and big line bundle.
    Let $E$ be a holomorphic vector bundle equipped with a singular Hermitian metric $h$.
    If $h$ is Griffiths semi-positive and $\nu(-\log \mathrm{det}\,h,x)<1$ for all points $x\in X$.
    Then there exists a singular Hermitian metric $h_L$ on $L$ such that $\mathscr{E}(h\otimes\mathrm{det}\,h\otimes h_L)\cong\mathcal{O}_X(E\otimes\mathrm{det}\,E\otimes L)$.
\end{corollary}

\begin{proof}
    By the assumption, i.e. $\nu(-\log (\mathrm{det}\,h)^2,x)<2$, the function $(\mathrm{det}\,h)^2$ is locally integrable. There exists a singular Hermitian metric $h_L$ such that $(\mathrm{det}\,h)^2\cdot h_L$ is locally integrable as in the proof of Proposition \ref{m.m.sheaf for Grif * nef big}.
    From $h\otimes\mathrm{det}\,h\otimes h_L=(\mathrm{det}\,h)^2\cdot h_L\cdot \widehat{h^*}$ and each element of $\widehat{h^*}$ is locally bounded [31,\,Lemma\,2.2.4], 
    for any local section $s\in C^{\infty}_x(E\otimes\mathrm{det}\,E\otimes L)$ the function $|s|^2_{h\otimes\mathrm{det}\,h\otimes h_L}$ is locally integrable. 
\end{proof}

Using Proposition \ref{m.m.sheaf for Grif * nef big} and Corollary \ref{m.m.sheaf for Grif * det Grif * nef big}, we obtain key lemmas and the proof of Theorem \ref{V-thm of nu(N) and Grif for nu(h)<1 near A in 7}.

\begin{lemma}\label{V-thm of nu(N) and L2-Nak for nu(h)<1 on X}
    Let $X$ be a projective manifold, $N$ be a nef line bundle and $E$ be a holomorphic vector bundle equipped with a singular Hermitian metric $h$.
    If $h$ is $L^2$-type Nakano semi-positive and that $\nu(-\log\mathrm{det}\,h,x)<2$ for all points $x\in X$, then we have 
    \begin{align*}
        H^q(X,K_X\otimes E\otimes N)=0 \quad \mathit{for} \quad q>n-\nu(N).
    \end{align*} 
\end{lemma}

\begin{proof}
    First suppose that $\nu(N)=n$, i.e. $N$ is big. 
    Since Proposition \ref{m.m.sheaf for Grif * nef big} and $N$ is nef and big, there exists a singular Hermitian metric $h_N$ such that $i\Theta_{N,h_N}\geq\delta\omega$ for some $\delta>0$ 
    and we get $\mathscr{E}(h\otimes h_N)\cong\mathcal{O}_X(E\otimes N)$, where $\omega$ is a \kah metric on $X$.
    Therefore, $h\otimes h_N$ is $L^2$-type strictly Nakano $\delta_\omega$-positive by $(b)$ of Corollary \ref{s d-posi * Nak semi-posi then s Nak d-posi}.
    From [19,\,Theorem\,1.5], for any $q>0$ we have the following cohomologies vanishing 
    \begin{align*}
        0=H^q(X,K_X\otimes\mathscr{E}(h\otimes h_N))\cong H^q(X,K_X\otimes E\otimes N).
    \end{align*}

    Now, if $\nu(N)<n$, we use hyperplane sections and argue by induction on $n=\mathrm{dim}\,X$.
    We can select a nonsingular ample divisor $A$ such that $\nu(N|_A)=\nu(N)$, where $\mathcal{O}_X(A)$ is positive. 
    Here, compactness of $X$, the line bundle $\mathcal{O}_X(A)\otimes N$ is also positive. And we get $\mathscr{E}(h)=\mathcal{O}_X(E)$ by the assumption $\nu(-\log\mathrm{det}\,h,x)<2$.
    
    Thus, from Theorem \ref{V-thm of L2-type Nakano posi on w.p.c} for any $q>0$, we get the following cohomologies vanishing
    \begin{align*}
        0=H^q(X,K_X\otimes A\otimes N\otimes \mathscr{E}(h))\cong H^q(X,K_X\otimes A\otimes N\otimes E).
    \end{align*} 
    And the exact sequence $0\longrightarrow K_X\longrightarrow K_X(\log A)=K_X\otimes\mathcal{O}_X(A)\longrightarrow K_A\longrightarrow 0$
    twisted by $\mathcal{O}_X(N\otimes E)$ yields an isomorphism
    \begin{align*}
        H^q(A,K_A\otimes (N\otimes E)|_A)\cong H^{q+1}(X,K_X\otimes N\otimes E) \quad \mathrm{for} \quad 0<q<n.
    \end{align*}

    Hence, by the induction hypothesis, i.e. $H^q(A,K_A\otimes (N\otimes E)|_A)=0$ for $q>n-1-\nu(N|_A)$, we have that 
    $H^q(X,K_X\otimes N\otimes E)=0$ for $q>n-\nu(N|_A)=n-\nu(N)$.
\end{proof}


\begin{lemma}\label{V-thm of nu(N) and Grif for nu(h)<1 on X}
    Let $X$ be a projective manifold, $N$ be a nef line bundle and $E$ be a holomorphic vector bundle equipped with a singular Hermitian metric $h$.
    If $h$ is Griffiths semi-positive and that $\nu(-\log\mathrm{det}\,h,x)<1$ for all points $x\in X$, then we have that 
    \begin{align*}
        H^q(X,K_X\otimes N\otimes E\otimes\mathrm{det}\,E)=0 \quad \mathit{for} \quad q>n-\nu(N).
    \end{align*} 
\end{lemma}

$\textit{Proof of Theorem \ref{V-thm of nu(N) and Grif for nu(h)<1 near A in 7}}$.
    By positivity of $\mathcal{O}_X(A)$ and compactness of $X$, $\mathcal{O}_X(A)\otimes N$ is also positive. By $(a)$ of Theorem \ref{V-thm of Griffiths semi-posi + (dual)-m-posi for (n,q)}, for any $q>0$ we get the cohomology vanishing 
    \begin{align*}
        H^q(X,K_X(\log A)\otimes N\otimes \mathscr{E}(h\otimes\mathrm{det}\,h))=H^q(X,K_X\otimes A\otimes N\otimes \mathscr{E}(h\otimes\mathrm{det}\,h))=0.
    \end{align*}

    Here, for any point $x\in A$, $(\mathrm{det}\,h)^2|_A$ is locally integrable near $x$ by Skoda's result \cite{Sko72}. 
    From the Ohsawa-Takegoshi $L^2$-extension theorem, $(\mathrm{det}\,h)^2$ is also locally integrable near $x$. Then we get $\mathscr{E}(h\otimes\mathrm{det}\,h)|_A=(E\otimes\mathrm{det}\,E)|_A$. 
    
    Therefore, from this and the short exact sequence of $K_X(\log A)$, the natural map
    \begin{align*}
        H^q(A,K_A\otimes (N\otimes E\otimes\mathrm{det}\,E)|_A)\longrightarrow H^{q+1}(X,K_X\otimes N\otimes \mathscr{E}(h\otimes\mathrm{det}\,h))
    \end{align*} 
    is an isomorphism for $q\geq1$ and is surjective for $q=0$. 
    
    By properties of plurisubharmonic and Definition \ref{def Griffiths semi-posi sing}, $h|_A$ is also Griffiths semi-positive over $A$. 
    Hence, form Lemma \ref{V-thm of nu(N) and Grif for nu(h)<1 on X}, we get 
    \begin{align*}
        H^q(A,K_A\otimes (N\otimes E\otimes\mathrm{det}\,E)|_A)=0
    \end{align*} 
    for $q>n-1-\nu(N|_A)$, where $\nu(N|_A)=\nu(N)<n$ by $N$ is not big. \qed

\vspace{2mm}

From the proof of this theorem, we immediately obtain the following.

\begin{corollary}\label{V-thm of big and Grif for nu(h)<1 near A}
    Let $X$ be a compact \kah manifold of dimension $n$, $N$ be a nef and big line bundle and $E$ be a holomorphic vector bundle equipped with a singular Hermitian metric $h$. 
    If $h$ is Griffiths semi-positive and there exists a smooth ample divisor $A$ such that $\nu(-\log\mathrm{det}\,h|_A,x)<1$ for all points in $A$, then for any $q>1$ we have that 
    \begin{align*}
        H^q(X,K_X\otimes N\otimes \mathscr{E}(h\otimes\mathrm{det}\,h))=0.
    \end{align*}
\end{corollary}

Finally, the proof of Theorem \ref{Fujita Conj for Grif} is obtained using the Castelnuovo-Mumford regularity and Theorem \ref{V-thm of Griffiths semi-posi + (dual)-m-posi for (n,q)} and \ref{V-thm of nu(N) and Grif for nu(h)<1 near A in 7}.

\vspace{2mm}

$\textit{Proof of Theorem \ref{Fujita Conj for Grif}}$.
    By 
    Lemma \ref{Castelnuovo-Mumford regularity}, we only need to prove $K_X\otimes L^{\otimes n}\otimes N\otimes \mathscr{E}(h\otimes\mathrm{det}\,h)$ is $0$-regular with respect to $L$.
    Hence, it suffices to show 
    \begin{align*}
        H^q(X,K_X\otimes L^{\otimes(n-q)}\otimes N\otimes \mathscr{E}(h\otimes\mathrm{det}\,h))=0 \quad \mathrm{for~\,all} \quad q>0.
    \end{align*}

    For $0<q<n$, by positivity of $L^{\otimes (n-q)}$ and compactness of $X$, $L^{\otimes (n-q)}\otimes N$ is also positive. Therefore, we have the desired vanishing cohomologies from Theorem \ref{V-thm of Griffiths semi-posi + (dual)-m-posi for (n,q)}.

    When $q=n$, we need to show $H^n(X,K_X\otimes N\otimes \mathscr{E}(h\otimes\mathrm{det}\,h))=0$. It is well-know that for a nef line bundle $N$, $\nu(N)=0$ if and only if $N$ is numerically trivial. This is assured by $\nu(N)\ne0$ and Theorem \ref{V-thm of nu(N) and Grif for nu(h)<1 near A in 7} and Corollary \ref{V-thm of big and Grif for nu(h)<1 near A}.

    The case when $q>n$ is obvious and we complete the proof.\qed

\vspace{2mm}

From the above proof and Lemma \ref{V-thm of nu(N) and L2-Nak for nu(h)<1 on X}, we obtain the following corollary.

\begin{corollary}
    Let $X$ be a compact \kah manifold of dimension $n$ and $E$ be a holomorphic vector bundle equipped with a singular Hermitian metric $h$. 
    Let $L$ be an ample and globally generated line bundle and $N$ be a nef but not numerically trivial line bundle.
    If $h$ is $L^2$-type Nakano semi-positive and that $\nu(-\log\mathrm{det}\,h,x)<2$ for all points $x\in X$, then the adjoint vector bundle
    $K_X\otimes L^{\otimes n}\otimes N\otimes E$ is globally generated.
\end{corollary}

\vspace{2mm}

{\bf Acknowledgement. } 
I would like to thank my supervisor Professor Shigeharu Takayama for guidance and helpful advice.

\end{document}